\documentclass[12pt]{article}
\usepackage{xcolor}
\usepackage{amsmath,amsthm,amsfonts,amssymb,amscd,caption,color,subcaption,cite}
\usepackage{graphicx}
\usepackage[mathscr]{eucal}
\allowdisplaybreaks[4]
\usepackage[figurename=Fig.]{caption}
\numberwithin{equation}{section}
\allowdisplaybreaks[4]
\newtheorem {Problem}{Problem}[section]
\newtheorem {Lemma}{Lemma}[section]
\newtheorem {Theorem} {Theorem}[section]

\newtheorem {Claim} {Claim}[section]
\usepackage{tikz}
\usetikzlibrary{positioning}
\usepackage{xcolor}
\usepackage{circuitikz}
\usetikzlibrary{patterns}

\begin{document}

\title{Ordinary and spectral extremal problems on vertex disjoint copies of even fans}

\author{Yiting Cai\footnote{E-mail: yitingcai@m.scnu.edu.cn},
Bo Zhou\footnote{
E-mail: zhoubo@m.scnu.edu.cn}\\
School of  Mathematical Sciences, South China Normal University,\\
Guangzhou 510631, P.R. China}

\date{}
\maketitle

\begin{abstract}
Let $\mathrm{ex}(n, F)$ and $\mathrm{spex}(n, F)$ be the maximum size and spectral radius among all $F$-free graphs with fixed order $n$, respectively. A fan is a graph $P_1\vee P_{s}$ (join of a vertex and a path of order $s$) for $s\ge 3$, and it is called an even fan if $s$ is even.
In this paper, we study $\mathrm{ex}(n,t(P_1\vee P_{2k}))$, $\mathrm{spex}(n,t(P_1\vee P_{2k}))$ with
$t\ge 1$ and $k\ge 3$ and characterize the corresponding extremal graphs for sufficiently large $n$.
\end{abstract}

\section{Introduction}

We consider simple and undirected graphs.
Let $G$ be a graph of order $n$ with vertex set $V(G)$ and edge set $E(G)$.
The size of $G$ is  $e(G)=|E(G)|$.
The adjacency matrix of $G$ is the $n\times n$ matrix $A(G)=(a_{uv})_{u,v\in V(G)}$, where, for any $u, v\in V(G)$,  $a_{uv}=1$ if $u$ and $v$ are adjacent and $0$ otherwise.
The spectral radius of $G$, denoted by $\rho(G)$, is the largest eigenvalue  of $A(G)$.

Given a graph $F$,  if a graph $G$ does not contain $F$ (as a subgraph), then $G$ is said to be $F$-free.
Let $\mathrm{ex}(n, F)$ be  the maximum size of graphs among all $F$-free graphs with fixed order $n$, and $\mathrm{Ex}(n, F)$  the set of $F$-free graphs of order $n$ with size $\mathrm{ex}(n, F)$.
Determining $\mathrm{ex}(n, F)$ and $\mathrm{Ex}(n, F)$ is typical in extremal combinatorics \cite{Fur}.
Let $\mathrm{spex}(n, F)$ be  the maximum  spectral radius of graphs among all $F$-free graphs with fixed order $n$, and $\mathrm{Spex}(n, F)$ the set of $F$-free graphs of order $n$ with spectral radius  $\mathrm{spex}(n, F)$. Determining $\mathrm{spex}(n, F)$ and $\mathrm{Spex}(n, F)$ was first proposed in generality in \cite{NV2}.

For vertex disjoint graphs $G$ and $H$, let $G\cup H$ denote the disjoint union of them. Denote by $tG$ the $t$ vertex-disjoint copies of $G$.
Denoted by $G\vee H$ the join of $G$ and $H$, which is obtained from $G\cup H$ by adding all possible edges between $V(G)$ and $V(H)$. Denote by $P_n$ the path of order $n$. A fan is a graph $P_1\vee P_{s}$ (join of a vertex and a path of order $s$) for $s\ge 3$, and it is called an even fan if $s$ is even.

In this paper, we study the following problem:

\begin{Problem} \label{pro}
For integers $t\ge 1$ and $k\ge 3$, determine
\begin{enumerate}
\item[(i)] $\mathrm{ex}(n, t(P_1\vee P_{2k}))$ and $\mathrm{Ex}(n, t(P_1\vee P_{2k}));$

\item[(ii)] $\mathrm{spex}(n, t(P_1\vee P_{2k}))$ and $\mathrm{Spex}(n, t(P_1\vee P_{2k}))$.
\end{enumerate}
\end{Problem}

For the case when $t=1$, Problem \ref{pro} has been settled in \cite{YLT,YLY}.
Instead of even fans, similar problems for paths \cite{CLZ, FYZ}, stars \cite{CLZ2}, odd wheels \cite{CDT,XZ,LND},  cycles \cite{Pik, Sim, FZL} and cliques \cite{Mo,NWK} have been studied.

Denote by $K_{n_1, n_2}$ the complete bipartite graph with classes of size $n_1$ and $n_2$. 
Let $(V_1, V_2)$ be its bipartition with $|V_i|=n_i$ for $i=1,2$.
A nearly $k$-regular graph is a graph in which  each vertex has degree $k$ except for one vertex with degree $k-1$.
Let $\mathcal{K}_{n_1,n_2}^{k-1}(P_{2k})$ be the set of graphs, in which each graph $G$ is obtained
from a copy of $K_{n_1,n_2}$ adding edges between vertices of $V_1$ so that $G[V_1]$ is a $P_{2k}$-free and is $(k-1)$-regular or nearly $(k-1)$-regular.
Fix integer $k\ge 3$. Yuan \cite{YLT} determined $\mathrm{ex}(n, P_1\vee C_{2k})$  for sufficiently  large $n$ and remarked on Page 704 that similar argument leads to $\mathrm{ex}(n, P_1\vee P_{2k})$.

\begin{Theorem}\label{T1}\cite{YLT,YLY}
Let $k$ be an integer with $k\ge3$. For sufficiently large $n$,
\[
\mathrm{ex}(n, P_1\vee P_{2k})=\max\left\{n_1n_2+\left\lfloor\frac{k-1}{2}n_1\right\rfloor: n_1+n_2=n, n_1\ge n_2\right\}
\]
and
\[
\mathrm{Ex}(n, P_1\vee P_{2k})=\mathcal{K}_{n_1, n_2}^{k-1}(P_{2k}) \mbox{ with $n_1\ge n_2$}.
\]
\end{Theorem}

\begin{Theorem}\label{T3}\cite{YLY}
Let $k$ be an integer with $k\ge3$. For sufficiently large $n$, if $G\in \mathrm{Spex}(n, P_1\vee P_{2k})$, then $G\in \mathcal{K}_{\frac{n}{2}+\tau, \frac{n}{2}-\tau}^{k-1}(P_{2k})$ with $|\tau|\le1$.
\end{Theorem}

We prove the following results.

\begin{Theorem}\label{N1}
Let $t$, $k$ be two  integers with $t\ge1$ and $k\ge3$. For sufficiently large $n$,
\begin{align*}
\mathrm{ex}(n, t(P_1\vee P_{2k}))={t-1\choose 2}+(n_1+t-1)(n-t+1)+\left\lfloor\frac{(k-1)n_1}{2}\right\rfloor-n_1^2
\end{align*}
and
\[
\mathrm{Ex}(n, t(P_1\vee P_{2k}))=\{K_{t-1}\vee H: H\in \mathcal{K}_{n_1,n-n_1-t+1}^{k-1}(P_{2k})\},
\]
where
\[
n_1=\begin{cases}
\left\lfloor\frac{n-t+1}{2}+\frac{k-1}{4}\right\rfloor & \mbox{if~}2n-2t+k\equiv 0 \!\!\pmod{4},\\[3mm]
\left\lceil\frac{n-t+1}{2}+\frac{k-1}{4}\right\rceil & \mbox{otherwise}.
\end{cases}
\]
\end{Theorem}

\begin{Theorem}\label{N2}
Let $t$, $k$ be two integers with $t\ge1$ and $k\ge3$. For sufficiently large $n$,
let $G\in \mathrm{Spex}(n, t(P_1\vee P_{2k}))$.
Then
$G=K_{t-1}\vee H$ with $H\in \mathcal{K}_{n_1,n-n_1-t+1}^{k-1}(P_{2k})$,
where if $t=1$, then
\[
n_1=
\begin{cases}
\frac{n}{2} & \mbox{if $k$ is odd and $n$ is even, or $k$ is even and $n\equiv 0\!\!\pmod{4}$},\\
\frac{n\pm1}{2} & \mbox{if $k$ and $n$ are both odd},\\
\frac{n-1}{2} & \mbox{if $k$ is even and $n\equiv 1\!\!\pmod{4}$},\\
\frac{n+1}{2} & \mbox{if $k$ is even and $n\equiv 3\!\!\pmod{4}$},\\
\frac{n}{2}\pm1 & \mbox{if $k$ is even and $n\equiv 2\!\!\pmod{4}$}.
\end{cases}
\]
and if $t\ge 2$, then
\[
n_1=
\begin{cases}
\left\lceil\frac{n-t+1}{2}\right\rceil & \mbox{if $k$ is odd, or
if $k$ is even and $n-t\equiv 2,3\!\!\pmod{4}$},\\
\left\lfloor\frac{n-t+1}{2}\right\rfloor & \mbox{if $k$ is even and $n-t\equiv 0\!\!\pmod{4}$},\\
\left\lceil\frac{n-t+2}{2}\right\rceil & \mbox{if $k$ is even and $n-t\equiv 1\!\!\pmod{4}$}.
\end{cases}
\]
\end{Theorem}

For odd $k\ge 3$ and $t\ge1$,  the extremal graphs in the previous two theorem
coincide if and only if $k=3$ and $n-t$ is even. For even $k\ge4$ and $t\ge1$, the extremal graphs in the previous two theorem
coincide if and only if $k=4$ and $n-t\equiv 1, 2\pmod{4}$, or $k=6$ and $n-t\equiv 1\pmod{4}$.

%

\section{Preliminaries}

%

For a graph $G$ with $u\in V(G)$, we denote by $N_G(u)$ and $d_G(u)$ the neighborhood and the degree of $u$ in $G$, respectively.
Let $\delta(G)$ and $\Delta(G)$ be the minimum and maximum degree in $G$, respectively.
For $\emptyset\ne U\subset V(G)$, let $G-U$ be the graph obtained from $G$ be removing the vertices in $U$ (and every edge incident to some vertex of $U$), if $U=\{u\}$, then we write $G-u$ for $G-\{u\}$.
And $G[U]$ denote the graph $G-(V(G)\setminus U)$. For $E_1\subseteq E(G)$, $G-E_1$ denotes the spanning subgraph of $G$ with edge set $E(G)\setminus E_1(G)$. If $E_1'$ is a subset set of the edge set of the complement of $G$, then $G+E_1'$ denotes the graph with $G$ being a spanning subgraph and edge set $E(G)\cup E_1'$. If $E'_1=\{uv\}$, then we write $G+uv$ for $G+\{uv\}$.
Let $K_n$ be the complete graph of order $n$. The vertex of degree $2k$ in $P_1\vee P_{2k}$ is its center.

\begin{Lemma}\label{T2}\cite{KST}
Let $n$, $a$, $b$ be three positive integers.  Then for some constant $c_{a,b}$,
\[
\mathrm{ex}(n, K_{a, b})\le c_{a,b}n^{2-\frac{1}{a}}.
\]
\end{Lemma}


For integer $r\ge 2$, $n_1,\dots, n_r\ge 1$, we denote by $K_{n_1,\dots, n_r}$ the complete $r$-partite graphs with classes of size $n_1,\dots, n_r$. If $n_1+\dots+n_r=n$ and $\max\{n_1,\dots, n_r\}-\min\{n_1,\dots, n_r\}=0,1$, then it is known as the Tur\'an graph, denoted by $T_{n,r}$.
Let $r\ge2$, $\frac{1}{\ln c}<c<r^{-8(r+21)(r+1)}$, $0<\epsilon<2^{-36}r^{-24}$ and let $G$ be a graph with order $n$.  Nikiforov \cite{VN} showed that
if $\rho(G)>(1-\frac{1}{r}-\epsilon)n$, then one of the following holds:
(i) $G$ contains a complete $r+1$-partite graph $K_{\lfloor c\ln n\rfloor, \dots, \lfloor c\ln n\rfloor, \lceil n^{1-\sqrt{c}}\rceil}$;
(ii) $G$ differs from $T_{n, r}$ in fewer than $(\epsilon^{\frac{1}{4}}+c^{\frac{1}{8r+8}})n^2$ edges.
From this result, Desai et al. deduced the following useful lemma.

\begin{Lemma}\label{co}\cite{DK}
Let $F$ be a graph with chromatic number $\chi(F)=r+1$. For every $\epsilon>0$, there exist $\delta>0$ and $n_0$ such that if $G$ is an $F$-free graph with order $n\ge n_0$ and $\rho(G)\ge(1-\frac{1}{r}-\delta)n$, then $G$ can be obtained from $T_{n, r}$ by adding and deleting at most $\epsilon n^2$ edges.
\end{Lemma}

\begin{Lemma}\label{cap}\cite{CF}
Let $A_1, \dots, A_n$ be finite sets. Then
\[
|A_1\cap\dots \cap A_n|\ge\sum_{i=1}^n|A_i|-(n-1)\left|\bigcup_{i=1}^n A_i\right|.
\]
\end{Lemma}

Let $\nu(G)$ be the matching number of a graph $G$, which is the number of edges in a maximum matching.

\begin{Lemma}\label{MD}\cite{CH}
Let $\nu$ and $\Delta$ be positive integers. For sufficiently large $n$, the number of edges in a graph of order $n$ with $\Delta(G)\le \Delta$ and  $\nu(G)\le \nu$ is at most
$\nu(\Delta+1)$.
\end{Lemma}

For a nonnegative square matrix $M$, denote by $\rho(M)$ the spectral radius of $M$.

For an $n\times n$ matrix $A$, whose rows and columns are indexed by elements in $X=\{1, \dots, n\}$,
let $\pi=\{X_1, \dots, X_s\}$ be a partition of $X$, and let $A_{ij}$ be the submatrix (block) of $A$ whose rows and columns are indexed by elements of $X_i$ and $X_j$ for $1\le i,j\le s$.
The characteristic matrix
$S$ is the $n\times s$ matrix whose $(i,j)$-entry is $1$ or $0$ according as $i\in X_j$ or not, where $1\le i\le n$ and $1\le j\le s$.
The quotient matrix of $A$ is the $s\times s$ matrix, whose $(i,j)$-entry is the average row sums of $A_{ij}$, where $1\le i,j\le s$.
Moreover, if the row sum of each  $A_{ij}$ is a constant, then we call the partition $\pi$ is equitable, and in this case $AS=SB$.

For a square nonnegative matrix  $M$, we denote by $\rho(M)$ the spectral radius of $M$.
So, for a graph $G$, $\rho(G)=\rho(A(G))$.

\begin{Lemma}\label{QM}\cite{BH}
Let $A$ be a real symmetric matrix. Then the spectrum of the quotient matrix $B$ of $A$
corresponding to an equitable partition is contained in the spectrum of $A$.
Furthermore, if $A$ is nonnegative and irreducible, then $\rho(A)=\rho(B)$, and
if $\mathbf{x}$ is a eigenvector of $B$ corresponding to $\rho(B)$, then $S\mathbf{x}$ is an eigenvector of $A$ for $\rho(A)$.
\end{Lemma}


\begin{Lemma}\cite[Corollary 2.1, p.~38]{Mi} \label{none}
Let $A$ and $B$ be two nonnegative matrix of order $n\times n$.
If $A\le B$  entry wise, then $\rho(A)\le \rho(B)$.
Moreover, if $A\ne B$ and $B$ is irreducible, then
$\rho(A)< \rho(B)$.
\end{Lemma}

For positive integer $r$ and $s$, we denote by $J_{r,s}$ the $r\times s$ all one's matrix.
For integer $s\ge 3$, the join of  vertex disjoint graphs $G_1,\dots, G_p$, denoted by
$G_1\vee\dots\vee G_s$, is defined as $(G_1\vee \dots G_{p-1})\vee G_p$.

\begin{Lemma}\label{jn}
Let $G_i$ be a graph with order $n_i$ and $\Delta(G_i)=d_i\ge 1$, where $i=1,\dots,s$, $s\ge 2$.
Then $\rho(G)\le\rho(D)$, where $G=G_1\vee\dots\vee G_s$ and
\[
D=\begin{pmatrix} d_1&n_2&\cdots&n_s\\n_1& d_2&\cdots&n_s\\ \vdots&\vdots&\ddots&\vdots\\ n_1&n_2&\cdots&d_s \end{pmatrix}.
\]
\end{Lemma}

\begin{proof} For $i\in\{1,\cdots,s\}$, any row sum of $A(G_i)$ is at most $d_i$, and if some row sum $t$ of $A(G_i)$ is less than $d_i$, then we form a matrix $B_i$ from $A(G_i)$ by changing the diagonal entry in each such  row from $0$ to $d_i-t$. Then  each $A_i$ has constant row sum $d_i$ for $i=1,2,3$. Evidently,
$A(G)\le B$  entry wise,
where
\[
B=\begin{pmatrix}
B_1& J_{n_1,n_2} & \cdots & J_{n_1,n_s}\\
J_{n_2,n_1}& B_2 & \cdots & J_{n_2, n_s}\\
\vdots & \vdots & \ddots & \vdots\\
J_{n_s,n_1} & J_{n_s, n_2} & \cdots & B_s
\end{pmatrix}.
\]
By Lemmas \ref{none} and \ref{QM}, $\rho(G)\le \rho(B)=\rho(D)$.
\end{proof}

Let $G$ be a given graph.
For $V_1\subset V(G)$ and $u\in V(G)$, let $N_{V_1}(u)=N_G(u)\cap V_1$ and $d_{V_1}(u)$ be the number of vertices in $N_{V_1}(u)$.
For a subgraph $H$ of $G$ induced by $U\subseteq V(G)$,  $e(U)=e(H)$.
For nonempty disjoint $S, T\in V(G)$, let $G[S, T]$ be the bipartite subgraph of $G$ on the vertex set $S\cup T$ and edge set $E(G[S, T])=\{uv\in E(G): u\in S, v\in T\}$. We also write  $e(S,T)$ or
$e(G[S], G[T])$ for $e(G[S, T])$. We use these notations in Sections 3 and 4.

\section{Proof of Theorem \ref{N1}}

Let $n_1$, $n_2$, $n$, $t$, $k$ be positive integers with $n_1\ge n_2$,  $n_1+n_2=n-t+1$, and $k\ge 3$.
Suppose that $\mathcal{K}_{n_1, n_2}^{k-1}(P_{2k})\ne \emptyset$. For $H\in  \mathcal{K}_{n_1, n_2}^{k-1}(P_{2k})$,
let
\[
f(n_1, n_2, t)=e(K_{t-1}\vee H).
\]
i.e.,
\begin{align*}
f(n_1, n_2, t)&={t-1\choose 2}+(t-1)(n_1+n_2)+\left\lfloor\frac{(k-1)n_1}{2}\right\rfloor+n_1n_2\\
&={t-1\choose 2}+(n_1+t-1)(n-t+1)+\left\lfloor\frac{(k-1)n_1}{2}\right\rfloor-n_1^2.
\end{align*}
By a direct calculation, $f(n_1, n_2, t)$ is maximized  if and only if
\[
n_1=\begin{cases}
\left\lfloor\frac{n-t+1}{2}+\frac{k-1}{4}\right\rfloor & \mbox{if~}2n-2t+k\equiv 0\!\!\pmod{4},\\
\left\lceil\frac{n-t+1}{2}+\frac{k-1}{4}\right\rceil & \mbox{otherwise}.
\end{cases}
\]
We denote by $f(n,t)$ the maximum of $f(n_1, n_2, t)$.

\begin{Lemma}\label{L1} Let $G$ be a $t(P_1\vee P_{2k})$-free graph of order $n$ with a partition $V(G)=V_1\cup V_2$, where $|V_1|\ge|V_2|$, and $|V_1|=\frac{n}{2}+O(\sqrt{n})$.
Suppose that $n$ is sufficiently large. If, for each $i\in\{1,2\}$, any nonempty set $S\subseteq V_i$
with $|S|\le t(k+1)$ satisfies
\[
\left|\bigcap_{v\in S} N_{V_{3-i}}(v)\right|\ge t(2k+1),
\]
then
\[
e(G)\le f(n, t)
\]
with  equality when $t\ge2$ only if $G$ has a vertex of degree $n-1$.
\end{Lemma}

\begin{proof}
Our proof proceeds by induction on  $t$.
The case $t=1$ is trivial by Theorem \ref{T1}.

Suppose that $t\ge 2$. 
Observe that
\begin{align}\label{e1}
f(n,t)-f(n-1, t-1)={t-1\choose 2}-{t-2\choose 2}+n-t+1=n-1.
\end{align}

\noindent{\bf Case 1.} For some $i=1,2$,  $\exists u\in V_i$ such that $d_{V_i}(u)\ge t(2k+1)$.

Let $G'=G-u$.  Suppose that  $G'$ contains $H:=(t-1)(P_1\vee P_{2k})$.  For $S\subseteq N_{V_i}(u)\setminus V(H)$ with $|S|=k$, as all vertices of $S\cup\{u\}$ have at least $t(2k+1)$ common neighbors in $V_{3-i}$, so  at least $2k+1>k$ of these neighbors lie in  $V_{3-i}\setminus V(H)$. Then
the vertices of $S\cup\{u\}$, together with $k$ neighbors of them in  $V_{3-i}\setminus V(H)$ will generate a copy of $P_1\vee P_{2k}$ with center $u$,  which  is vertex disjoint from $H$ in $G$.
So $G$ contains $t(P_1\vee P_{2k})$, a contradiction.
It follows that $G'$ is $(t-1)(P_1\vee P_{2k})$-free.
Let $V_1'=V_2$ and $V_2'=V_1\setminus\{u\}$ if $i=1$ and $|V_1|=|V_2|$, and $V_i'=V_i\setminus\{u\}$ and $V_{3-i}'=V_{3-i}$ otherwise. Then for each $j=1,2$, any nonempty set $S\subseteq V_j'$
with $|S|\le (t-1)(k+1)$ satisfies
\[
\left|\bigcap_{v\in S} N_{V_{3-j}'}(v)\right|\ge\left|\bigcap_{v\in S} N_{V_{3-j}}(v)\setminus\{u\}\right|\ge t(2k+1)-1>(t-1)(2k-1).
\]
By the induction hypothesis and using \eqref{e1}, we have
\[
e(G')\le f(n-1,t-1)= f(n,t)-(n-1).
\]
As  $e(G)=d_G(u)+e(G')\le n-1+e(G')$, we have $e(G)\le f(n,t)$ with equality only if $d(u)=n-1$.

\noindent{\bf Case 2.} For any $i=1,2$, $\forall u\in V_i$,  $d_{V_i}(u)<t(2k+1)$.

In this case, we have $d_G(u)<|V_{3-i}|+t(2k+1)$. 
Assume that $G$ contains one copy of $P_1\vee P_{2k}$, as otherwise the result follows  from the fact that $e(G)\le f(n, 1)< f(n, t)$.
Let $U$ be the vertex set of this  $P_1\vee P_{2k}$, consisting of $r$ vertices from $V_1$ and $s$ vertices from $V_2$, where $r+s=2k+1$.
Let $G'=G-U$.
Then $G'$ is $(t-1)(P_1\vee P_{2k})$-free of order $n-2k-1$.

Let $V_1'=V_2\setminus\{U\}$ and $V_2'=V_1\setminus\{U\}$ if $i=1$ and $|V_1|-r<|V_2|-s$, and $V_i'=V_i\setminus\{U\}$ and $V_{3-i}'=V_{3-i}\setminus\{U\}$ otherwise. Then for each $j=1,2$, any nonempty set $S\subseteq V_j'$
with $|S|\le (t-1)(k+1)$ satisfies
\[
\left|\bigcap_{v\in S} N_{V'_{3-j}}(v)\right|\ge\left|\bigcap_{v\in S} N_{V_{3-j}}(v)\setminus U\right|>(t-1)(2k-1).
\]
By the induction hypothesis,
we have
\[
e(G')\le f(n-2k-1, t-1).
\]
So
\begin{align}
e(G)&\le e(G')+r(|V_2|+t(2k+1))+s(|V_1|+t(2k+1)) \notag \\
&\le f(n-2k-1, t-1)+r|V_2|+s|V_1|+t(2k+1)^2. \label{e0}
\end{align}
Note that
\begin{align}\label{e2}
&\quad f(n,t)-f(n-2k,t)   \notag \\
&=(n_1+t-1)(n-t+1)+\left\lfloor\frac{(k-1)n_1}{2}\right\rfloor-n_1^2-(n_1-k+t-1)(n-2k-t+1)  \notag \\
&\quad -\left\lfloor\frac{(k-1)(n_1-k)}{2}\right\rfloor+(n_1-k)^2  \notag \\
&> kn+k(t-k-1),
\end{align}
where the last inequality holds as $\left\lfloor\frac{(k-1)n_1}{2}\right\rfloor-\left\lfloor\frac{(k-1)(n_1-k)}{2}\right\rfloor>0$.
From  \eqref{e1} and \eqref{e2}, we have
\begin{align*}
f(n-2k-1, t-1)&= f(n-2k, t)-(n-2k-1)\\
&< f(n,t)-kn-k(t-k-1)-(n-2k-1)\\
&=f(n,t)-(k+1)n-k(t-k-3)+1.
\end{align*}
Now from \eqref{e0} and noting that $|V_1|=\frac{n}{2}+O(\sqrt{n})\ge|V_2|$, and $r+s=2k+1$, we have
\begin{align*}
e(G)&\le f(n-2k-1, t-1)+r|V_2|+s|V_1|+t(2k+1)^2\\
&< f(n,t)-(k+1)n-k(t-k-3)+1+(r+s)|V_1|+t(2k+1)^2\\
&=f(n,t)-(k+1)n-k(t-k-3)+1+\frac{2k+1}{2}n+t(2k+1)^2+O(\sqrt{n})\\
&=f(n,t)-\frac{1}{2}n+O(\sqrt{n})\\
&<f(n,t)
\end{align*}
as $n$ is sufficiently large.
\end{proof}

We need the progressive induction introduced by Simonovits \cite{MS}.

\begin{Lemma} \label{ind} \cite{MS}
Let $\mathcal{G}_n$  be a family of disjoint graphs of order $n$ with the same size and let $\mathcal{G}=\bigcup_{n\ge 1}\mathcal{G}_n$. Let $B$ be a condition defined on  $\mathcal{G}$.
Let $\phi$ be a function on
graphs in $\mathcal{G}$ satisfying (a) $\phi(G)$ is a non-negative integer, and if $G$ satisfies $B$, then
$\phi(G)=0$, and (b) there is an $n_0$  such that if  $G\in \mathcal{G}_n$ with $n>n_0$, then either $G$ satisfies $B$ or
there is  an $n'$ with $\frac{n}{2}<n'<n$ and a graph $G'\in \mathcal{G}_{n'}$ such that  $\phi(G)<\phi(G')$. Then  there exists $n_1$ such that if $n>n_1$, then it follows  from $G\in \mathcal{G}$ that $G$ satisfies $B$.
\end{Lemma}

Now we are ready to prove Theorem~\ref{N1}.

\begin{proof}[Proof of Theorem~\ref{N1}]
Let $\mathcal{H}_{n,t}=\{K_{t-1}\vee H: H\in \mathcal{K}_{n_1,n-n_1-t+1}^{k-1}(P_{2k})\}$ and  $F=t(P_1\vee P_{2k})$.
From the initial discussion in this part, we have

\begin{Claim} \label{c1}
For any $G_0\in \mathcal{H}_{n,t}$, $e(G_0)=f(n,t)$.
\end{Claim}

We want to show that
$\mathrm{ex}(n, F)=f(n,t)$ and
$Ex(n, F)=\mathcal{H}_{n,t}$.
We prove the result  by induction on $t$.
The case $t=1$ is trivial by Theorem \ref{T1}.
Suppose next that $t\ge2$ and it holds for $t-1$.
By Claim \ref{c1},  a graph in $\mathcal{H}_{n,t}$ has size $f(n,t)$.
Let  $G$ be a $F$-free  graph of order $n$ that maximizes the size.
It suffices to show that, for sufficiently large $n$,
$G\in \mathcal{H}_{n,t}$.

Let $\mathcal{G}_n=\mathrm{Ex}(n, F)$. Let $B$ be the condition that $G\in \mathcal{H}_{n,t}$, where  $G\in \mathcal{G}_n$. Let $\phi(G)=e(G)-f(n,t)$, that is, $\phi(G)=\mathrm{ex}(n,F)-f(n,t)$, which depends only on the order $n$ of $G$, so we write $\phi(n)$ for $\phi(G)$.

Let $N$ be an even positive integer, which is large enough but less than $\frac{n}{15}$.
As $G$ is $F$-free, we have by Lemma \ref{T2} that
there is an integer $n_0$ depending on $N$ such that for $n>n_0$,
\[
e(G)> e(K_{\lceil\frac{n}{2}\rceil, \lfloor\frac{n}{2}\rfloor})=\left\lfloor\frac{n^2}{4}\right\rfloor>\mathrm{ex}(n, K_{N,N}),
\]
implying that $G$ contains $K_{N, N}$.
Let $(X, Y)$ be the bipartition of $K_{N, N}$.
Let $G_1=G[X\cup Y]$ and $G_2=G[V(G)\setminus(X\cup Y)]$. Then
\[
e(G)=e(G_1)+e(G_2)+e(G_1, G_2).
\]
Similarly, we choose a graph  $G^*\in \mathcal{H}_{n,t}$ by the following way:
Let $G_1^*\in\mathcal{K}_{N, N}^{k-1}(P_{2k})$. It  is a subgraph of some graph in $H^*\in \mathcal{K}_{n_1, n-n_1-t+1}^{k-1}(P_{2k})$.
Let $G^*=K_{t-1}\vee H^*$. Let  $G_2^*=G^*-V(G_1^*)$.
 By Claim \ref{c1},
$e(G_2^*)=e(K_{t-1}\vee(H^*-V(G_1^*)))=f(n-2N,t)$.
It is easy to see that
$e(G_1^*, G_2^*)=(n-2N-t+1)N+2N(t-1)=N(n-2N+t-1)$. Then
\[
f(n,t)=e(G^*)=e(G_1^*)+f(n-2N,t)+N(n-2N+t-1).
\]

\noindent
{\bf Case 1.} $\phi(n)<\phi(n-1)$ or $\phi(n)<\phi(n-2N)$.

Note that $\frac{n}{2}<n-1<n$ and $\frac{n}{2}<n-2N<n$ as  $n>4N$.
By Lemma \ref{ind}, for some $N_1$ with $n>N_1$,  $G\in \mathcal{H}_{n,t}$ as $G\in \mathcal{G}_n$. So for sufficiently large $n$, $G\in \mathcal{H}_{n,t}$, as desired.

\noindent
{\bf Case 2.}  $\phi(n)\ge \phi(n-1)$ and $\phi(n)\ge \phi(n-2N)$.

\begin{Claim} \label{yy}
$tN\ge N(n-2N+t-1)-e(G_1, G_2)$.
\end{Claim}

\begin{proof}
By Lemma \ref{L1}, we have
\begin{align*}
e(G_1)&\le f(2N, t)\\
&\le {t-1\choose 2}+\left(\frac{2N-t+1}{2}+\frac{k-1}{4}+t-1\right)(2N-t+1)\\
&\quad +\frac{k-1}{2}\left(\frac{2N-t+1}{2}+\frac{k-1}{4}\right)-\left(\frac{2N-t+1}{2}+\frac{k-1}{4}\right)^2\\
&= {t-1\choose 2}+N^2+\frac{2t+k-3}{2}N-\frac{(6t-k-5)(2t+k-3)}{16},
\end{align*}
so
\begin{align*}
e(G_1)-e(G_1^*) &\le f(2N, t)-\left(N^2+\frac{(k-1)N}{2}\right)\\
&\le  {t-1\choose 2}+(t-1)N-\frac{(6t-k-5)(2t+k-3)}{16}\\
& \le tN
\end{align*}
as $N$ is large enough.
Thus
\begin{align*}
0&\le  e(G)-f(n,t)-(\mathrm{ex}(n-2N, F)-f(n-2N,t))\\
& \le  e(G)-f(n,t)-(e(G_2)-f(n-2N,t))\\
&= e(G_1)- e(G_1^*)+ e(G_1, G_2)-N(n-2N+t-1)\\
&\le tN+e(G_1, G_2)-N(n-2N+t-1),
\end{align*}
implying the desired statement.
\end{proof}

\begin{Claim} \label{ddd1} $\delta(G)\ge \frac{n}{2}$.
\end{Claim}

\begin{proof}
Suppose by contradiction that  $\exists u\in V(G)$ such that $d_G(u)<\frac{n}{2}$.
Let
\[
n_1'=\begin{cases}
\left\lfloor\frac{n-t}{2}+\frac{k-1}{4}\right\rfloor & \mbox{if~}2n-2t+k\equiv 2\!\!\pmod{4},\\
\left\lceil\frac{n-t}{2}+\frac{k-1}{4}\right\rceil & \mbox{otherwise}.
\end{cases}
\]
Then
\begin{align*}
f(n,t)-f(n-1,t)&=(n_1+t-1)(n-t+1)+\left\lfloor\frac{(k-1)n_1}{2}\right\rfloor-n_1^2\\
&\quad -(n_1'+t-1)(n-t)-\left\lfloor\frac{(k-1)n_1'}{2}\right\rfloor+n_1'^2.
\end{align*}
It is easy to see that $n_1'=n_1$ or $n_1'=n_1-1$.
In the former case,
\[
f(n,t)-f(n-1,t)=n_1+t-1\ge\frac{n-t+1}{2}+t-1>\frac{n}{2},
\]
and in the latter case,
\begin{align*}
f(n,t)-f(n-1,t)&=n-n_1+\left\lfloor\frac{(k-1)n_1}{2}\right\rfloor-\left\lfloor\frac{(k-1)(n_1-1)}{2}\right\rfloor\\
&\ge n-n_1+\frac{k-1}{2}-\frac{1}{2}\\
&\ge n-\frac{2n-2t+k+3}{4}+\frac{k-1}{2}-\frac{1}{2}\\
&>\frac{n}{2}.
\end{align*}
So we have $f(n,t)-f(n-1,t)>\frac{n}{2}$ in either case. As $G-u$ is $F$-free, we have
\[
e(G)=e(G-u)+d_{G}(u)< \mathrm{ex}(n-1,F)+\frac{n}{2}<\mathrm{ex}(n-1,F)+f(n,t)-f(n-1,t).
\]
So
\[
\phi(n)=e(G)-f(n,t)<\mathrm{ex}(n-1, F)-f(n-1,t)=\phi(n-1),
\]
a contradiction.
\end{proof}


\begin{Claim}\label{c4}
If  $G[N(v)]$ contains $K_{t(2k+1), t(2k+1)}$ for some $v\in V(G)$, then
$G\in \mathcal{H}_{n,t}$.
\end{Claim}

\begin{proof}
Suppose that $G-v$ contains  $(t-1)(P_1\vee P_{2k})$.
Note that  $G[N(v)]$ contains $K_{t(2k+1),t(2k+1)}$, in which
there is a copy of  $K_{2k+1, 2k+1}$ so a copy of $K_{k, k}$ that is vertex disjoint from $(t-1)(P_1\vee P_{2k})$.
So $G$ contains $F$, a contradiction. It follows that $G-v$ is $(t-1)(P_1\vee P_{2k})$-free.

By the inductive hypothesis, for sufficiently large $n$,
\[
\mathrm{ex}(n-1, (t-1)(P_1\vee P_{2k}))=f(n-1,t-1) \mbox{ and }
\mathrm{Ex}(n-1, (t-1)(P_1\vee P_{2k}))=\mathcal{H}_{n-1,t-1},
\]
so by (\ref{e1}),
\begin{align*}
f(n,t) & \le e(G)= d_G(v)+e(G-v)\le n-1+e(G-v)\\
&\le n-1+\mathrm{ex}(n-1, (t-1)(P_1\vee P_{2k}))
=n-1+f(n-1, t-1)=f(n,t),
\end{align*}
from which we have $e(G)=f(n,t)$, $d_G(v)=n-1$, $e(G-v)=f(n-1,t-1)$, and  $G-v\in \mathcal{H}_{n-1,t-1}$,
so $G\in \mathcal{H}_{n,t}$.
\end{proof}

If  $d_X(v_0)\ge t(2k+1)$ and $d_Y(v_0)\ge t(2k+1)$ for some $v_0\in V(G)$, then $G[N(v_0)]$ contains $K_{t(2k+1),t(2k+1)}$, so Claim \ref{c4} implies that for sufficiently large $n$, $G\in \mathcal{H}_{n,t}$, as desired.

Suppose  that for any $v\in V(G)$, $d_X(v)< t(2k+1)$ or $d_Y(v)< t(2k+1)$.
Let
\[
X_1=\{v\in V(G)\setminus(X\cup Y): d_X(v)<t(2k+1),  d_Y(v)>N-2t(2k+1)\},\]
\[
Y_1=\{v\in V(G)\setminus(X\cup Y): d_Y(v)<t(2k+1),  d_X(v)>N-2t(2k+1)\}
\]
and
\[
Z=V(G)\setminus(X\cup Y\cup X_1\cup Y_1).
\]

Suppose that there exists $u\in X\cup X_1$ such that $d_{X\cup X_1}(u)\ge t(2k+1)$ or $v\in Y\cup Y_1$ such that $d_{Y\cup Y_1}(v)\ge t(2k+1)$.
By Lemma \ref{cap}, any $t(2k+1)+1$ vertices in $X\cup X_1$ (or $Y\cup Y_1$) have at least
\[
 (t(2k+1)+1)(N-2t(2k+1)+1)-t(2k+1)N=N-(t(2k+1)+1)(2t(2k+1)-1)
\]
common neighbors in $Y$ (or $X$). As $N$  is large enough, this number can be larger than $t(2k+1)$.
So $G[N(u)]$ (or $G[N(v)]$) contains $K_{t(2k+1),t(2k+1)}$.
By Claim \ref{c4}, $G\in \mathcal{H}_{n,t}$  for sufficiently large $n$.

Suppose in the following that $d_{X\cup X_1}(u)<t(2k+1)$ for any $u\in X\cup X_1$, $d_{Y\cup Y_1}(v)<t(2k+1)$ for any $v\in Y\cup Y_1$.

\begin{Claim}\label{c5}
$|Z|<\frac{15}{7}N$.
\end{Claim}

\begin{proof}
Since
\[
\mbox{$d_{X\cup X_1}(u)<t(2k+1)$ for any $u\in X\cup X_1$}
\]
and
\[
\mbox{$d_{Y\cup Y_1}(v)<t(2k+1)$ for any $v\in Y\cup Y_1$},
\]
we have
\[
e(X,X_1)+e(Y,Y_1)<2Nt(2k+1).
\]
Trivially,
\[
e(X,Y_1)+e(Y,X_1)\le (|X_1|+|Y_1|)N.
\]
As for any $u\in Z$,
\[
d_{X\cup Y}(u)=d_X(u)+d_Y(u)<t(2k+1)+N-2t(2k+1)=N-t(2k+1),
\]
we have
\[
e(X\cup Y,Z)\le |Z|(N-t(2k+1)).
\]
It follows that
\begin{align*}
e(G_1, G_2)&= e(X,X_1)+e(Y,Y_1)+e(X,Y_1)+e(Y,X_1)+e(X\cup Y,Z)\\
&<2Nt(2k+1)+(|X_1|+|Y_1|)N+|Z|(N-t(2k+1))\\
&= (2N-|Z|)t(2k+1)+N(n-2N)\\
&= 4tkN+(n-2N+t-1)N+(t+1)N-|Z|t(2k+1).
\end{align*}
By Claim \ref{yy},
we have
\[
tN\ge (n-2N+t-1)N-e(G_1, G_2)>-4tkN-(t+1)N+|Z|t(2k+1).
\]
So
\[
|Z|<\frac{1}{t(2k+1)}(4tkN+2tN+N)<\frac{4k+3}{2k+1}N \le \frac{15}{7}N. \qedhere
\]
\end{proof}

\begin{Claim}\label{c6}
$|X\cup X_1|=\frac{n}{2}+O(\sqrt{n})$ and $|Y\cup Y_1|=\frac{n}{2}+O(\sqrt{n})$.
\end{Claim}

\begin{proof}
By Claim \ref{c5}, $|Z|<\frac{15}{7}N$, so the number of edges incident to vertices in  $Z$ is $O(n)$.
Recall that  $d_{X\cup X_1}(u)<t(2k+1)$ for any vertex $u\in X\cup X_1$ and $d_{Y\cup Y_1}(v)<t(2k+1)$ for any $v\in Y\cup Y_1$.
So $e(G[X\cup X_1])+e(G[Y\cup Y_1])=O(n)$.
Let
\[
G'=G-E(G[X\cup X_1])-E(G[Y\cup Y_1])-E(Z, V(G)),
 \]
which is a bipartite graph with size $\lfloor\frac{n^2}{4}\rfloor-O(n)$.
Then there exists a constant $N_2$ such that $||X\cup X_1|-\frac{n}{2}|\le N_2\sqrt{n}$ and $||Y\cup Y_1|-\frac{n}{2}|\le N_2\sqrt{n}$.
So $|X\cup X_1|=\frac{n}{2}+O(\sqrt{n})$ and $|Y\cup Y_1|=\frac{n}{2}+O(\sqrt{n})$.
\end{proof}

For any $v\in Z$, we have  either $d_{X\cup X_1}(v)\ge\frac{n}{6}$ or $d_{Y\cup Y_1}(v)\ge\frac{n}{6}$, as otherwise,
\[
d_G(v)<d_{X\cup X_1}(v)+d_{Y\cup Y_1}(v)+d_Z(v)\le \frac{n}{6}\cdot 2+|Z|-1<\frac{n}{3}+\frac{15}{7}N-1<\frac{n}{2}
\]
as $n$ is sufficiently large and $N<\frac{n}{15}$, contradicting Claim \ref{ddd1}.
By Claim \ref{c6}, we have $|Z|=n-|X\cup X_1|-|Y\cup Y_1|=O(\sqrt{n})$.
So for any $v\in Y\cup Y_1$,
\[
d_{X\cup X_1}(v)=d_G(v)-d_{Y\cup Y_1}(v)-d_Z(v)>\frac{n}{2}-t(2k+1)-|Z|= \frac{n}{2}-O(\sqrt{n}).
\]
Similarly,
\[
d_{Y\cup Y_1}(v)> \frac{n}{2}-O(\sqrt{n}) \mbox{ for any $v\in X\cup X_1$}.
\]
Suppose that  $v_1\in Z$  with  $d_{X\cup X_1}(v_1)\ge\frac{n}{6}$ such that $d_{Y\cup Y_1}(v_1)\ge t(2k+1)$.  By Lemma \ref{cap}, and Claim \ref{c6}, for any $v\in N_{Y\cup Y_1}(v_1)$,
\begin{align*}
|N_{X\cup X_1}(v)\cap N_{X\cup X_1}(v_1)|&\ge
|N_{X\cup X_1}(v)|+|N_{X\cup X_1}(v_1)|-|X\cup X_1|\\
&> \frac{n}{2}-O(\sqrt{n})+d_{X\cup X_1}(v_1)-\left(\frac{n}{2}+O(\sqrt{n})\right)\\
&=d_{X\cup X_1}(v_1)-O(\sqrt{n}).
\end{align*}
So, by Lemma \ref{cap}, any $t(2k+1)$ vertices in $N_{Y\cup Y_1}(v_1)$ have more than
\[
t(2k+1)(d_{X\cup X_1}(v_1)-O(\sqrt{n}))-(t(2k+1)-1)d_{X\cup X_1}(v_1)
\ge \frac{n}{6}-O(\sqrt{n})>t(2k+1)
\]
common neighbors in $N_{X\cup X_1}(v_1)$ as $n$ is sufficiently large, then
there is a copy of $K_{t(2k+1), t(2k+1)}$ in $G[N(v_1)]$. Then the result follows from Claim \ref{c4}.
The proof is similar if $v_1\in Z$  with  $d_{Y\cup Y_1}(v_1)\ge\frac{n}{6}$ such that $d_{X\cup X_1}(v_1)\ge t(2k+1)$.

We are left with the case that 
$Z=Z_1\cup Z_2$ with
\[
Z_1=\left\{v\in V(G)\setminus(X\cup Y\cup X_1\cup Y_1): d_{X\cup X_1}(v)<t(2k+1), d_{Y\cup Y_1}(v)\ge\frac{n}{6}\right\}
 \]
 and
\[
Z_2=\left\{v\in V(G)\setminus(X\cup Y\cup X_1\cup Y_1): d_{Y\cup Y_1}(v)<t(2k+1), d_{X\cup X_1}(v)\ge\frac{n}{6}\right\}.
\]
%
%
Assume that $|X\cup X_1\cup Z_1|\ge|Y\cup Y_1\cup Z_2|$.
Let $V_1=X\cup X_1\cup Z_1$ and $V_2=Y\cup Y_1\cup Z_2$.
For any $v\in V_i$, 
we have
\[
d_{V_{3-i}}(v)=d_G(v)-d_{V_i}(v)\ge\frac{n}{2}-(t(2k+1)+O(\sqrt{n}))=\frac{n}{2}-O(\sqrt{n}).
 \]
By Lemma \ref{cap}, for any nonempty set $S\subseteq V_i$ with $|S|\le t(k+1)$,
\begin{align*}
\left|\bigcap_{v\in S} N_{V_{3-i}}(v)\right|& \ge |S|\left(\frac{n}{2}-O(\sqrt{n})\right)-(|S|-1)|V_{3-i}|\\
&\ge |S|\left(\frac{n}{2}-O(\sqrt{n})\right)-(|S|-1)\left(\frac{n}{2}+O(\sqrt{n})\right)\\
&\ge\frac{n}{2}+O(\sqrt{n})\\
&>t(2k+1)
\end{align*}
for  sufficiently large $n$.
Then by Lemma \ref{L1},  $f(n,t)\le e(G)\le f(n,t)$, so there exists a vertex of degree $n-1$ in $G$, which is impossible as $d_{V_i\setminus Z_i}(u)\le t(2k+1)$ for any $u\in V_i$ and $|V_i\setminus Z_i|=\frac{n}{2}+O(\sqrt{n})$ by Claim \ref{c6}. This completes the proof.
\end{proof}

\section{Proof of Theorem \ref{N2}}

\begin{Lemma} \label{FINa}
Let $t$, $k$ be two  integers with $t\ge1$ and $k\ge3$. For sufficiently large $n$,
\begin{align*}
\frac{n}{4}\left(n+k+2t-\frac{7}{2}\right)<\mathrm{ex}(n, t(P_1\vee P_{2k}))<\frac{n^2}{4}+\frac{2t+k-2}{4}n.
\end{align*}
\end{Lemma}

\begin{proof} Let
\[
m_1=\begin{cases}
\left\lfloor\frac{n-t+1}{2}+\frac{k-1}{4}\right\rfloor & \mbox{if~}2n-2t+k\equiv 0\!\!\pmod{4},\\
\left\lceil\frac{n-t+1}{2}+\frac{k-1}{4}\right\rceil & \mbox{otherwise}.
\end{cases}
\]
 By Theorem \ref{N1},
\begin{align*}
&\quad\mathrm{ex}(n, t(P_1\vee P_{2k}))\\
&= \frac{t^2-3t+2}{2}+(m_1+t-1)(n-t+1)+\left\lfloor\frac{(k-1)m_1}{2}\right\rfloor-m_1^2\\
&> \frac{t^2-3t+2}{2}+\frac{(n+t-1)(n-t+1)}{2}+\frac{(k-1)(n-t+1)}{4}-\frac{1}{2}-\left(\frac{n-t+1}{2}\right)^2\\
&=\frac{1}{4}n^2+\frac{1}{4}(k+2t-3)n-\frac{1}{4}(k+t-1)t+\frac{k}{4}-\frac{1}{2}\\
&>\frac{n}{4}\left(n+k+2t-\frac{7}{2}\right),
\end{align*}
where the first inequality holds as $m_1>\frac{n-t+1}{2}$
and the last inequality holds as $n$ is sufficiently large.

Viewed as a function of $m_1$, $\mathrm{ex}(n, t(P_1\vee P_{2k}))$ maximizes at $m_1=\frac{n-t+1}{2}+\frac{k-1}{4}$, so
\begin{align*}
\mathrm{ex}(n, t(P_1\vee P_{2k}))&\le \frac{t^2-3t+2}{2}+\frac{1}{4}\left(2n+2t+k-3\right)\left(n-t+1\right)\\
&\quad +\frac{k-1}{2}\cdot\frac{2n-2t+k+1}{4}-\frac{1}{16}\left(2n-2t+k+1\right)^2\\
&= \frac{n^2}{4}+\frac{2t+k-3}{4}n-\frac{(t+k-1)t}{4}+\frac{(k+1)^2}{16}\\
&<\frac{n^2}{4}+\frac{2t+k-2}{4}n,
\end{align*}
where the last inequality follows as $n$ is sufficiently large.
\end{proof}

\begin{Lemma}
Theorem~\ref{N2} is true for $t=1$.
\end{Lemma}

\begin{proof} Theorem \cite{YLY} states that
$G\in \mathcal{K}_{\frac{n}{2}+\tau, \frac{n}{2}-\tau}^{k-1}(P_{2k})$ with $|\tau|\le 1$.
It remains only to determine the value of $\tau\in \{0, \pm \frac{1}{2}, \pm 1\}$.
Note that $G$ can be written as $G=G_1\vee(\frac{n}{2}-\tau)K_1$, where $|V(G_1)|=\frac{n}{2}+\tau$, and $G_1$ is $(k-1)$-regular if $(k-1)|V(G_1)|$ is even and nearly $(k-1)$-regular otherwise.
Let
\[
A_\tau=\begin{pmatrix}
k-1 & \frac{n}{2}-\tau\\ \frac{n}{2}+\tau & 0\end{pmatrix}.
\]
By a direct calculation, the largest eigenvalue of $A_\tau$ is
\[
\rho(A_\tau)=\frac{1}{2}\left(k-1+\sqrt{-4\tau^2+n^2+(k-1)^2}\right).
\]
Let $f(x)=-4x^2+n^2+(k-1)^2$, then it is an even function and maximizes at $x=0$.
So
\begin{align}\label{111}
\rho(A_0)>\rho(A_{\frac{1}{2}})=\rho(A_{-\frac{1}{2}})>\rho(A_1)=\rho(A_{-1}).
\end{align}
If $G_1$ is $(k-1)$-regular, then $A_\tau$ is the quotient matrix of $A(G)$ corresponding to the equitable partition $V(G)=V(G_1)\cup V((\frac{n}{2}-\tau)K_1)$,  so $\rho(G)=\rho(A_\tau)$ by Lemma \ref{QM}.

If $k$ is odd,
then  $G_1$ is $(k-1)$-regular and $\rho(G)=\rho(A_\tau)$, so by (\ref{111}), we have
\[
\tau=
\begin{cases}
0 & \mbox{if $n$ is even},\\
\pm \frac{1}{2} & \mbox{otherwise}.
\end{cases}
\]

Suppose next that $k$ is even.

Suppose that $n\equiv 0\!\!\pmod{4}$. Then $\tau=0$ or $\pm 1$.
If $\tau\ne 0$, then $G_1$ is nearly $(k-1)$-regular, so we have by Lemmas \ref{jn} and \ref{none}, together with (\ref{111}),  that $\rho(G)<\rho(A_\tau)<\rho(A_0)$, while if $\tau=0$, then $G_1$ is $(k-1)$-regular and $\rho(G)=\rho(A_0)$. This shows that
$\tau=0$.

Suppose that $n\equiv 1\!\!\pmod{4}$.  Then $\tau=\pm\frac{1}{2}$. If $\tau=\frac{1}{2}$, then $G_1$ is nearly $(k-1)$-regular, so by Lemmas \ref{jn} and \ref{none}, together with (\ref{111}), we have $\rho(G)<\rho(A_{\frac{1}{2}})=\rho(A_{-\frac{1}{2}})$, while
if $\tau=-\frac{1}{2}$, then $G_1$ is $(k-1)$-regular and $\rho(G)=\rho(A_{-\frac{1}{2}})$.
This shows that $\tau=-\frac{1}{2}$.

Suppose that $n\equiv 2\!\!\pmod{4}$. Then $\tau=0$ or $\pm1$. If $\tau=\pm1$, then $G_1$ is $(k-1)$-regular and $\rho(G)=\rho(A_1)=\rho(A_{-1})$.
Suppose  that $\tau=0$. Then $G_1$ is nearly $(k-1)$-regular,  so $\sum_{v\in V(G_1)}d_G(v)=\frac{n}{2}\left(\frac{n}{2}+k-1\right)-1=\frac{n^2}{4}+\frac{(k-1)n}{2}-1$.
Let $G_2=G-V(G_1)$.
Then for any $v\in V(G_2)$, $d_G(v)=\frac{n}{2}$, so $\sum_{v\in V(G_2)}d_G(v)=\frac{n^2}{4}$.
It follows that
\[
\rho(G)\ge\frac{1}{n}\sum\limits_{v\in V(G)}d_G(v)=\frac{1}{n}\left(\frac{n^2}{4}+\frac{(k-1)n}{2}-1+\frac{n^2}{4}\right)=\frac{n+k-1}{2}-\frac{1}{n}.
\]
Let $\mathbf{x}$ be a positive eigenvector associated with $\rho(G)$, which can be chosen so that $\max\{x_u: u\in V(G)\}=1$.
Assume that $x_{u_1}=1$.
Note that for any $u,v\in V(G_2)$, $x_u=x_v$.
Assume that $x_1:=\max\{x_v: v\in V(G_1)\}$.
As for $u\in V(G_2)$, $\rho(G)x_u=\sum_{v\in V(G_1)}x_v\le\frac{n}{2}x_1$, we have $x_u\le\frac{n}{2\rho(G)}x_1<x_1$.
It follows that $u_1\in V(G_1)$.
Fix a vertex $u_2\in V(G_2)$.
Since $\rho(G)=\rho(G)x_{u_1}\le\sum\limits_{v\in V(G_2)}x_v+\sum\limits_{v\in N_{G_1}(u_1)}x_v\le\frac{n}{2}x_{u_2}+k-1$, we have
\[
x_{u_2}\ge\frac{2(\rho(G)-k+1)}{n}\ge\frac{n-k+1}{n}-\frac{2}{n^2}=1-\frac{k-1}{n}-\frac{2}{n^2}.
\]
Assume that $d_{G_1}(v_0)=k-2$ for $v_0\in V(G_1)$.
Note that the matching number of $G_1$ are at least $\min\left\{\delta(G_1), \frac{|V(G_1)|}{2}\right\}=k-2$.
Let $\{v_1v_1',\dots,v_{\frac{k-2}{2}}v_{\frac{k-2}{2}}'\}$ be a matching of $G_1$ such that $d_{G_1}(v_i)=d_{G_1}(v_i')=k-1$ for $i=1,\cdots,\frac{k-2}{2}$, and
let $V_1'=\left\{v_i, v_i': 1\le i\le \frac{k-2}{2}\right\}$.
Let
\begin{align*}
G'&=G-\left\{v_iv_i': 1\le i\le \frac{k-2}{2}\right\}-\{u_2w: w\in V(G_1)\setminus (V_1'\cup\{v_0\})\}\\
&\quad+\{u_2w: w\in V(G_2)\setminus\{u_2\}\}.
\end{align*}
Then $G'\in\mathcal{K}_{\frac{n}{2}+1, \frac{n}{2}-1}^{k-1}(P_{2k})$, so $G'$ is $(P_1\vee P_{2k})$-free, and then $A_1$ is the quotient matrix of $G'$ corresponding to the partition $V(G')=V((V(G_1)\cup\{u_2\}) \cup (V(G_2)\setminus\{u_2\}))$.
So by Lemma \ref{QM},
\[
\rho(G')=\rho(A_1)=\frac{1}{2}\left(k-1+\sqrt{n^2+(k-1)^2-4}\right).
\]
Let
\[
S=\begin{pmatrix}
1 & \cdots & 1 & 0 & \cdots & 0\\
0 & \cdots & 0 & 1 & \cdots & 1\end{pmatrix}^\top
\]
be a $n\times 2$ matrix with the first $\frac{n}{2}+1$ elements in the first column are $1$.
Then we can check that $A(G')S=SA_1$.
Let $(y_1, y_2)^\top$ be a positive eigenvector associated with $\rho(A_1)$.
Then by Lemma \ref{QM},
\[
\mathbf{y}:=S(y_1, y_2)^\top=(\underbrace{y_1, \cdots, y_1}_{\frac{n}{2}+1}, \underbrace{y_2, \cdots, y_2}_{\frac{n}{2}-1})^\top
\]
is an eigenvector associated with $\rho(G')$.
Note that
\[
\rho(A_1)y_1=(k-1)y_1+\left(\frac{n}{2}-1\right)y_2,
\]
and
\begin{align*}
&\quad \frac{1}{2}\left(k+1+\sqrt{n^2+(k-1)^2-4}\right)\left(1-\frac{k-1}{n}-\frac{2}{n^2}\right)
-\frac{n}{2}+1\\
&=\frac{1}{2}\left(k+3+\sqrt{n^2+(k-1)^2-4}\right)-\frac{n}{2}-\frac{k-1}{2}
{\left(\frac{k+1}{n}+\sqrt{1+\frac{(k-1)^2-4}{n^2}}\right)}\\
&\quad -\frac{\sqrt{n^2+(k-1)^2-4}}{n^2}-\frac{k+1}{n^2} \\
&= \frac{k+3}{2}-\frac{k-1}{2}-o(n)\\
&> 0
\end{align*}
as $n$ is sufficiently large.
Then
\begin{align*}
&\quad \mathbf{y}^\top(\rho(G')-\rho(G))\mathbf{x}=\mathbf{y}^\top(A(G')-A(G))\mathbf{x}\\
&=\sum\limits_{w\in V(G_2)\setminus\{u_2\}}(x_{u_2}y_2+y_1x_w)-\sum\limits_{w\in V(G_1)\setminus(V_1'\cup\{v_0\})}(x_{u_2}y_1+y_1x_w)
-\sum\limits_{i=1}^{\frac{k-2}{2}}(x_{v_i}y_1+y_1x_{v_i'})\\
&\ge \left(\frac{n}{2}-1\right)(x_{u_2}y_2+x_{u_2}y_1)-\left(\frac{n}{2}-k+1\right)(x_{u_2}y_1+y_1)-(k-2)y_1\\
&= \left((k-2)y_1+\Big(\frac{n}{2}-1\Big)y_2\right)x_{u_2}-\left(\frac{n}{2}-1\right)y_1\\
&= \left(\rho(A_1)+1\right)y_1 x_{u_2}-\left(\frac{n}{2}-1\right)y_1\\
&= \left(\frac{1}{2}\left(k+1+\sqrt{n^2+(k-1)^2-4}\right)x_{u_2}-\frac{n}{2}+1\right)y_1\\
&\ge \left(\frac{1}{2}\left(k+1+\sqrt{n^2+(k-1)^2-4}\right)\left(1-\frac{k-1}{n}-\frac{2}{n^2}\right)
-\frac{n}{2}+1\right)y_1 \\
&> 0.
\end{align*}
So $\rho(A_1)=\rho(G')>\rho(G)$. It follows that $\tau=\pm 1$.


 Suppose that $n\equiv 3 \!\!\pmod{4}$. Then $\tau=\pm\frac{1}{2}$.  If $\tau=-\frac{1}{2}$, then $G_1$ is nearly $(k-1)$-regular, so by Lemmas \ref{jn} and \ref{none}, and (\ref{111}), we have $\rho(G)<\rho(A_{-\frac{1}{2}})=\rho(A_{\frac{1}{2}})$, while if $\tau=\frac{1}{2}$, then $G_1$ is $(k-1)$-regular and $\rho(G)=\rho(A_{\frac{1}{2}})$. So $\tau=\frac{1}{2}$.
\end{proof}

\begin{Lemma}
Theorem~\ref{N2} is true for $t\ge 2$.
\end{Lemma}

\begin{proof}
Let $F=t(P_1\vee P_{2k})$.
Suppose that $G$ is a $F$-free with order $n$ that maximizes the spectral radius, where $n$ is sufficiently large.
Suppose that $G$ is disconnected. Then, for some component $G_1$ of $G$, $\rho(G)=\rho (G_1)$. Let $G_2$ be a different component. The graph $G'$ obtained from $G$ by adding an edge between one vertex of $G_1$ and one vertex of $G_2$ is also $F$-free with order $n$. However, we have by Lemma \ref{L1} that $\rho(G')>\rho(G_1\cup G_2)=\rho(G_1)=\rho(G)$, a contradiction. It follows that $G$ is connected.
By Perron-Frobenius Theorem, $A(G)$ has a positive eigenvector $\mathbf{x}=(x_1, \dots, x_n)^\top$ associated with $\rho(G)$, which can be chosen so that $\max\{x_w: w\in V(G)\}=1$.
Assume that $x_{u}=1$.
Let $\eta=\frac{1}{9(3k+1)}$ be a constant.
It suffices to show that $G\cong K_{t-1}\vee H$ with  $H\in \mathcal{K}_{n_1,n-n_1-t+1}^{k-1}(P_{2k})$.

\begin{Claim}\label{C21}
$\rho(G)>\frac{n+k+2t}{2}-\frac{7}{4}$.
\end{Claim}

\begin{proof} Let $G_0\in Ex(n, F)$.  By Lemma \ref{FINa},
\[
\rho(G)\ge \rho(G_0)\ge \frac{2\mathrm{ex}(n, F)}{n}>\frac{n+k+2t}{2}-\frac{7}{4}. \qedhere
\]
\end{proof}

\begin{Claim}\label{C22}
$e(G)\ge\frac{n^2}{4}-\frac{1}{4}\eta^2n^2$.  
Furthermore, if $G$ has a partition $V(G)=V_1\cup V_2$ such that $e(V_1, V_2)$ is maximum, then $e(V_1)+e(V_2)<\frac{1}{4}\eta^2n^2$, and for each $i\in\{1,2\}$, $\frac{n}{2}-\eta n<|V_i|<\frac{n}{2}+\eta n$.
\end{Claim}

\begin{proof}
Let $\epsilon$ be a positive constant with $\epsilon<\frac{1}{4}\eta^2$. Note that $\chi(F)=3=2+1$ and $G$ is $F$-free.  By Claim \ref{C21} $\rho(G)>\frac{n+k+2t}{2}-\frac{7}{4}\ge (1-\frac{1}{2}-\delta)n$ for any $\delta>0$.
As $n$ is sufficiently large, we have by Lemma \ref{co} that
\[
e(G)\ge e(T_{n, 2})-\epsilon n^2= \left \lfloor \frac{n^2}{4}\right\rfloor-\epsilon n^2\ge\frac{n^2}{4}-\frac{1}{4}\eta^2n^2
\]
and there exists a partition $V(G)=U_1\cup U_2$ such that $\lfloor\frac{n}{2}\rfloor\le|U_1|\le|U_2|\le\lceil\frac{n}{2}\rceil$ and $e(U_1)+e(U_2)\le \epsilon n^2 <\frac{1}{4}\eta^2n^2$.
Let $V(G)=V_1\cup V_2$ such that $e(V_1, V_2)$ is  maximum.
Then
\[
e(V_1)+e(V_2)\le e(U_1)+e(U_2)<\frac{1}{4}\eta^2n^2.
\]
Assuming that $|V_1|=\frac{n}{2}+a$ and $|V_2|=\frac{n}{2}-a$, we have
\[
\frac{n^2}{4}-\frac{1}{4}\eta^2n^2\le e(G)\le|V_1||V_2|+e(V_1)+e(V_2)< \left(\frac{n}{2}+a\right)\left(\frac{n}{2}-a\right)+\frac{1}{4}\eta^2n^2,
\]
so $a^2<\frac{1}{2}\eta^2n^2$, then we have $-\eta n<a<\eta n$.
Thus, $\frac{n}{2}-\eta n<|V_i|<\frac{n}{2}+\eta n$ for $i=1, 2$.
\end{proof}

Let $U=\{v\in V(G): d_G(v)\le(\frac{1}{2}-2\eta)n\}$.

\begin{Claim}\label{C23}
 $|U|\le\frac{1}{2}\eta n$.
\end{Claim}

\begin{proof}
Suppose that $|U|>\frac{1}{2}\eta n$.
Choose $U'\subset U$ such that $|U'|=\lfloor\frac{1}{2}\eta n\rfloor$.
By Lemma \ref{FINa},
\[
\mathrm{ex}(n, F)<\frac{n^2}{4}+\frac{2t+k-2}{4}n.
\]
By Claim \ref{C22}, and Lemma \ref{FINa},
\begin{align*}
e(G\setminus U')&\ge  e(G)-\sum_{v\in U'}d_G(v)\\
&\ge \frac{n^2}{4}-\frac{1}{4}\eta^2n^2-\left(\frac{1}{2}-2\eta\right)n\cdot\left\lfloor\frac{\eta n}{2}\right\rfloor\\
&\ge \frac{1}{4}(1-\eta+3\eta^2)n^2\\
&= \frac{1}{4}\left(n-\frac{1}{2}\eta n+\frac{1}{2}\right)^2
-\frac{1}{4}\left(n-\frac{1}{2}\eta n\right)+\frac{1}{16}(11\eta^2n^2-1)\\
&\ge\frac{(n-\lfloor\frac{1}{2}\eta n\rfloor)^2}{4}
-\frac{1}{4}\left(n-\left\lfloor\frac{1}{2}\eta n\right\rfloor\right)+\frac{1}{16}(11\eta^2n^2-1)\\
&>  \frac{(n-\lfloor\frac{1}{2}\eta n\rfloor)^2}{4}+\frac{2t+k-2}{4}\left(n-\left\lfloor\frac{1}{2}\eta n\right\rfloor\right)\\
&>ex\left(n-\left\lfloor\frac{1}{2}\eta n\right\rfloor, F\right),
\end{align*}
where the penultimate inequality follows as $n$ is sufficiently large.
Thus  $G\setminus U'$ contains $F$, a contradiction.
\end{proof}

Let $W=W_1\cup W_2$ with  $W_i=\{v\in V_i: d_{V_i}(v)\ge 2\eta n\}$ for $i=1, 2$.
\begin{Claim}\label{C24}
$|W|<\frac{1}{4}\eta n$.
\end{Claim}

\begin{proof}
For each $i=1,2$, we have
\[
2e(V_i)=\sum_{v\in V_i}d_{V_i}(v)\ge\sum_{v\in W_i}d_{V_i}(v)\ge |W_i|\cdot 2\eta n,
\]
so
\[
e(V_1)+e(V_2)\ge|W|\eta n.
\]
By Claim \ref{C22}, we have $|W|\eta n<\frac{1}{4}\eta^2n^2$, i.e.,  $|W|<\frac{1}{4}\eta n$.
\end{proof}

\begin{Claim}\label{C25}
For any $v\in V_i\setminus(U\cup W)$ with $i=1,2$,
\[
d_{V_i\setminus(U\cup W)}(v)\le (t-1)(2k+1)+k-1.
\]
\end{Claim}

\begin{proof} Suppose that there exists $v_0\in V_i\setminus(U\cup W)$ such that $d_{V_i\setminus(U\cup W)}(v_0)\ge (t-1)(2k+1)+k$.
Since $G\in \mathrm{Spex}(n, F)$, $G$ contains $H:=(t-1)(P_1\vee P_{2k})$.
Let $V_i'=V_i\setminus(U\cup W\cup V(H))$ for  $i=1,2$.
Hence, $d_{V_i'}(v_0)\ge d_{V_i\setminus(U\cup W)}(v_0)-|V(H)|\ge k$.
So $G[N_{V_i'}(v_0)\cup\{v_0\}]$ contains $K_{1,k}$.
Next, we show that $G-(U\cup W\cup V(H))$ contains $P_1\vee P_{2k}$.

For any $v\in V_i'$, we have $d_G(v)>(\frac{1}{2}-2\eta)n$ as $v\notin U$ and $d_{V_i}(v)< 2\eta n$ as $v\notin W$.
Then
\[
d_{V_{3-i}}(v)=d_G(v)-d_{V_i}(v)>\left(\frac{1}{2}-2\eta\right)n-2\eta n=\frac{n}{2}-4\eta n.
\]
As $n$ is sufficiently large, we may require that $|V(H)|=(t-1)(2k+1)<\frac{1}{4}\eta n$.
By Claims \ref{C23} and \ref{C24},
\[
d_{V_{3-i}'}(v)\ge d_{V_{3-i}}(v)-|U|-|W|-|V(H)|>\frac{n}{2}-4\eta n-\frac{1}{2}\eta n-\frac{1}{4}\eta n-\frac{1}{4}\eta n=\frac{n}{2}-5\eta n.
\]
Assume that $V(K_{1,k})=\{v_0, \dots, v_k\}$. 
Then by Lemma \ref{cap} and Claim \ref{C22}, we have
\begin{align*}
|N_{V_{3-i}'}(v_0)\cap\dots\cap N_{V_{3-i}'}(v_k)| &\ge |N_{V_{3-i}'}(v_0)|+\dots+|N_{V_{3-i}'}(v_k)|-k|V_{3-i}|\\
&>(k+1)\left(\frac{n}{2}-5\eta n\right)-k\left(\frac{n}{2}+\eta n\right)\\
&= \frac{n}{2}-(6k+5)\eta n\\
&=\left(\frac{1}{2}-\frac{6k+5}{9(3k+1)}\right)n\\
&> k
\end{align*}
as $n$ is sufficiently large.
Therefore, there has a copy of $H':=P_1\vee P_{2k}$ with center $v_0$ in $G-(U\cup W\cup V(H))$.
As $V(H)\cap V(H')=\emptyset$, $G$ contains $H\cup H'=F$, a contradiction.
\end{proof}

\begin{Claim}\label{C26}
Let $R\subset V(G)$ with $|R|\le (t-1)(2k+1)$.
If there exists $u_0\in W\setminus U$, then $G-((U\cup W\cup R)\setminus\{u_0\})$ contains $P_1\vee P_{2k}$.
\end{Claim}

\begin{proof}
Suppose that $u_0\in W\setminus U$ and $u_0\in V_1$. Then $d_G(u_0)>(\frac{1}{2}-2\eta)n$ as $u_0\notin U$ and $d_{V_1}(u_0)\ge 2\eta n$ as $u_0\in W$.
Let $V_i'=V_i\setminus(U\cup W\cup R)$ for $i=1, 2$.
By Claims \ref{C23} and \ref{C24},
\begin{align*}
d_{V_1'}(u_0)&\ge d_{V_1}(u_0)-|U|-|W|-|R|\\
&> 2\eta n-\frac{1}{2}\eta n-\frac{1}{4}\eta n-(t-1)(2k+1)\\
&> \eta n
>k
\end{align*}
as $n$ is sufficiently large. So we can
choose $k$ vertices $\{u_1, \dots, u_k\}\subseteq N_{V_1'}(u_0)$.
Since $V(G)=V_1\cup V_2$ is a partition of $V(G)$ such that $e(V_1, V_2)$ is maximum,
$d_{V_2}(u_0)\ge\frac{1}{2}d_G(u_0)>(\frac{1}{4}-\eta)n$.
By Claims \ref{C23} and \ref{C24}, we have
\begin{align*}
d_{V_2'}(u_0)&\ge  d_{V_2}(u_0)-|U|-|W|-|R|\\
&> \left(\frac{1}{4}-\eta\right)n-\frac{1}{2}\eta n-\frac{1}{4}\eta n-(t-1)(2k+1)\\
&> \frac{n}{4}-2\eta n
\end{align*}
as $n$ is sufficiently large.
From the proof of Claim \ref{C25},
we have $d_{V_2'}(u_j)>\frac{n}{2}-5\eta n$ for each $j\in\{1,\dots,k\}$.
Recall that $\eta=\frac{1}{9(3k+1)}$.
By Lemma \ref{cap}, we have
\begin{align*}
|N_{V_2'}(u_0)\cap\dots\cap N_{V_2'}(u_k)|&\ge |N_{V_2'}(u_0)|+\dots+|N_{V_2'}(u_k)|-k|V_2|\\
&>  \frac{n}{4}-2\eta n+k\left(\frac{n}{2}-5\eta n\right)-k\left(\frac{n}{2}+\eta n\right)\\
&= \frac{n}{4}-(6k+2)\eta n\\
&=\left(\frac{1}{4}-\frac{6k+2}{9(3k+1)}\right)n\\
&=\frac{1}{36}n>k
\end{align*}
as $n$ is sufficiently large.
So there has a copy of $P_1\vee P_{2k}$ with center $u_0$ in $G-(U\cup W\cup R)\setminus\{u_0\}$.
\end{proof}

\begin{Claim}\label{C266}
$|W\setminus U|\le t-1$.
\end{Claim}

\begin{proof}
Suppose that $|W\setminus U|\ge t$.
Let $v_1\in W\setminus U$.  By Claim \ref{C26}, there is a subgraph $H_1:= P_1\vee P_{2k}$ in $G-((U\cup W)\setminus \{v_1\})$.
We repeat this process to find $v_1,\dots, v_t$ such that for $i=2,\dots,t$,
$v_i\in W\setminus (U\cup\{v_1, \dots, v_{i-1}\})$, and we have by Claim \ref{C26} that
$G-((U\cup W\cup \cup_{j=2}^{i} V(H_{j-1}))\setminus \{v_i\})$ contains  $H_i:=P_1\vee P_{2k}$.
Thus  $G$ contains $F$, a contradiction. It follows that $|W\setminus U|\le t-1$.
\end{proof}

Let $x_{v^*}=\max\{x_v: v\in V(G)\setminus W\}$.

\begin{Claim}\label{C2666}
$x_{v^*}\ge\frac{\rho(G)-|W|}{n-|W|}>\frac{1-\eta}{2}+\frac{2k-3}{4n}$, and $v^*\notin U$.
\end{Claim}
\begin{proof}
Note that
\[
\rho(G) x_u\le\sum_{v\in W}x_v+\sum_{v\in V(G)\setminus W}x_v\le|W|x_u+(n-|W|)x_{v^*}.
 \]
By Claims \ref{C21}, \ref{C23} and \ref{C266},
\[
x_{v^*}\ge\frac{\rho(G)-|W|}{n-|W|}>\frac{\frac{n+k+2t}{2}-\frac{7}{4}-(t-1)-|U|}{n}\ge\frac{1-\eta}{2}+\frac{2k-3}{4n}.
\]

Since
\begin{align*}
\rho(G) x_{v^*}&=\sum_{v\in N_W(v^*)}x_v+\sum_{v\in N_{V(G)\setminus W}(v^*)}x_v\\
&\le d_W(v^*)+(d_G(v^*)-d_W(v^*))x_{v^*}\\
&\le |W|(1-x_{v^*})+d_G(v^*)x_{v^*},
\end{align*}
we have by Claims \ref{C21} and \ref{C24} that
\begin{align*}
d_G(v^*)&\ge\rho(G)-|W|\left(\frac{1}{x_{v^*}}-1\right)\\
&\ge \frac{n+k+2t}{2}-\frac{7}{4}-\frac{1}{4}\eta n\left(\frac{4n}{2(1-\eta)n+2k-3}-1\right)\\
&= \frac{n}{2}+\frac{2k+4t-7}{4}-\frac{1}{4}\eta n\left(\frac{4}{2(1-\eta)+\frac{2k-3}{n}}-1\right)\\
&> \frac{n}{2}-2\eta n,
\end{align*}
so $v^*\notin U$.
\end{proof}

For convenience, let $V_i'=V_i\setminus (U\cup W)$ for $i=1,2$.

\begin{Claim}\label{C277}
For $v\in V(G)$, let $G'$ be the graph obtained from $G$ by deleting all edges incident to $v$ and adding either all possible edges between $v$ and vertices in $V_1'$, or all possible edges between $v$ and vertices in $V_2'$, but not both.
Then  $G'$ is $F$-free.
\end{Claim}

\begin{proof} Fix $j\in \{1,2\}$ so that all possible edges between $v$ and vertices in $V_j'$ are added to form $G'$.
Suppose that $G'$ contains $F$.
Then $v\in V(F)$.
Take $H_1:=P_1\vee P_{2k}$ from $F$ with $v\in V(H_1)$.
Let $V(H_1)=\{v, v_1, \dots, v_{2k}\}$ and $H_2=F-V(H_1)$.
Then $H_2$ is a subgraph of $G$ as any added edge to form $G'$ from $G$ lies outside $H_2$.
For $i\in\{1,\dots,2k\}$, $d_G(v_i)>(\frac{1}{2}-2\eta)n$ as $v_i\notin U$, and by Claim \ref{C25}, $d_{V_j'}(v_i)\le(t-1)(2k+1)+k-1<t(2k+1)<\frac{1}{4}\eta n$ as $n$ is sufficiently large.
By Claims \ref{C23} and  \ref{C24}, $|U|+|W|< \frac{3}{4}\eta n$. So
for each $i\in\{1,\dots,2k\}$,
\[
|N_{V_{3-j}'}(v_i)|\ge d_G(v_i)-d_{V_j'}(v_i)-|U|-|W|>\frac{n}{2}-3\eta n.
\]

\noindent{\bf Case i.} $v$ is a center of $H_1$.

Since $G$ is $F$-free, $v$ is not adjacent to all vertices of $V(H_1)\setminus\{v\}$ in $G$.
By Lemma \ref{cap} and Claim \ref{C22},
\begin{align*}
|N_{V_{3-j}'}(v_1)\cap\dots\cap N_{V_{3-j}'}(v_{2k})|&\ge |N_{V_{3-j}'}(v_1)|+\dots+|N_{V_{3-j}'}(v_{2k})|-(2k-1)|V_{3-j}'|\\
&\ge 2k\left(\frac{n}{2}-3\eta n\right)-(2k-1)|V_{3-j}|\\
&> 2k\left(\frac{n}{2}-3\eta n\right)-(2k-1)\left(\frac{n}{2}+\eta n\right)\\
&= \left(\frac{1}{2}-\frac{8k-1}{9(3k+1)}\right)n\\
&>(t-1)(2k+1)
\end{align*}
as $n$ is sufficiently large,
so
\[
|N_{V_{3-j}'}(v_1)\cap\dots\cap N_{V_{3-j}'}(v_{2k})|-|V(H_2)|\ge 1,
\]
implying that for some $v_0\in V_{3-j}'$, $v_0\notin V(H_2)$ and $v_0\in N_{V_{3-j}'}(v_1)\cap\dots\cap N_{V_{3-j}'}(v_{2k})$. Then $G[\{v_0, v_1,\dots, v_{2k}\}]=P_1\vee P_{2k}$, so $G$ contains $F$, a contradiction.

\noindent{\bf Case ii.} $v$ is an end vertex of $P_{2k}$ in $H_1$.

Assume that $v_{2k}$ is the center of $H_1$ and $v$ is adjacent to $v_1$ in $H_1$.
Then by Lemma \ref{cap} and Claim \ref{C22},
\begin{align*}
|N_{V_{3-j}'}(v_1)\cap N_{V_{3-j}'}(v_{2k})|&\ge |N_{V_{3-j}'}(v_1)|+|N_{V_{3-j}'}(v_{2k})|-|V_{3-j}|\\
&>  2\left(\frac{n}{2}-3\eta n\right)-\left(\frac{n}{2}+\eta n\right)\\
&= \left(\frac{1}{2}-\frac{7}{9(3k+1)}\right)n\\
&>(t-1)(2k+1)
\end{align*}
as $n$ is sufficiently large.
So we can find $v_0\in V_{3-j}'$ such that $v_0\notin V(H_2)$ and $v_0\in N_{V_{3-j}'}(v_1)\cap N_{V_{3-j}'}(v_{2k})$.
Similarly as in Case i, $G$ contains $F$, a contradiction.

\noindent{\bf Case iii.} $v$ is an internal vertex of $P_{2k}$ in $H_1$.

Assume that $v_{\ell}$ is the center of $H_1$ and $v$ is adjacent to $v_1$ and $v_2$, where $3\le\ell\le2k$.
By Lemma \ref{cap} and Claim \ref{C22},
\begin{align*}
&\quad |N_{V_{3-j}'}(v_1)\cap N_{V_{3-j}'}(v_2)\cap N_{V_{3-j}'}(v_\ell)|\\
&\ge |N_{V_{3-j}'}(v_1)|+|N_{V_{3-j}'}(v_2)|+ |N_{V_{3-j}'}(v_\ell)|-2|V_{3-j}|\\
&=\left(\frac{1}{2}-\frac{11}{9(3k+1)}\right)n\\
&>(t-1)(2k+1)
\end{align*}
as $n$ is sufficiently large.
Similarly, $G$ contains $F$, a contradiction.

By combining Cases i--iii, $G'$ is $F$-free.
\end{proof}

\begin{Claim}\label{C27}
$U=\emptyset$.
\end{Claim}

\begin{proof}
Suppose that $U\ne\emptyset$.  Let $v\in U$.
Then $d_G(v)\le(\frac{1}{2}-2\eta)n$.
Assume that $v^*\in V_1\setminus W$.
Let $G'$ be the graph obtained from $G$ by deleting all edges incident to $v$ and adding all possible edges between $v$ and vertices in $V_2'$.
By Claim \ref{C277}, $G'$ is $F$-free.
By Claim \ref{C2666}, $v^*\notin U$, and by Claim \ref{C25}, we have $d_{V_1\setminus (W\cup U)}(v^*)\le(t-1)(2k+1)+k-1<t(2k+1)$.
Then by Claim \ref{C23}, we have
\begin{align*}
\rho(G) x_{v^*} & \le\sum_{w\in W\cup U}x_w+\sum_{w\in N_{V_1'}(v^*)}x_w+\sum_{w\in V_2'}x_w\\
&<|W|+|U|x_{v^*}+t(2k+1)x_{v^*}+\sum_{w\in V_2'}x_w\\
&\le|W|+\left(\frac{1}{2}\eta n+t(2k+1)\right)x_{v^*}+\sum_{w\in V_2'}x_w,
\end{align*}
so
\[
\sum_{w\in V_2'}x_w>(\rho(G)-\frac{1}{2}\eta n-t(2k+1))x_{v^*}-|W|.
\]
Note that
\[
\sum_{w\in N(v)}x_w=\sum_{w\in N_W(v)}x_w+\sum_{w\in N_{V(G)\setminus W}(v)}x_w\le |W|+d_G(v)x_{v^*}.
\]
Then, by Rayleigh's principle and Claims \ref{C21}, \ref{C24} and \ref{C2666}, we have
\begin{align*}
&\quad \rho(G')-\rho(G)\\
&\ge \frac{2x_v}{\mathbf{x}^\top\mathbf{x}}\left(\sum_{w\in V_2'}x_w-\sum_{w\in N(v)}x_w\right)\\
&> \frac{2x_v}{\mathbf{x}^\top\mathbf{x}}\left((\rho(G)-\frac{1}{2}\eta n
-t(2k+1))x_{v^*}-|W|-|W|-d_G(v)x_{v^*}\right)\\
&= \frac{2x_v}{\mathbf{x}^\top\mathbf{x}}\left((\rho(G)-\frac{1}{2}\eta n-t(2k+1)-d_G(v))x_{v^*}-2|W|\right)\\
&>  \frac{2x_v}{\mathbf{x}^\top\mathbf{x}}\left(\left(\rho(G)-\frac{1}{2}\eta n-t(2k+1)-\frac{n}{2}+2\eta n\right)
\left(\frac{1-\eta}{2}+\frac{2k-3}{4n}\right)-\frac{1}{2}\eta n\right)\\
&> \frac{x_v}{\mathbf{x}^\top\mathbf{x}}\left(\left(\frac{3}{2}\eta n-t(2k+1)\right)
\left(1-\eta+\frac{2k-3}{2n}\right)-\eta n\right)\\
&>  0,
\end{align*}
where the last inequality
follows as $n$ is sufficiently large,
so we have a contradiction.
\end{proof}

By  Claim \ref{C27}, we have $V_i'=V_i\setminus W$ for $i=1,2$.

\begin{Claim}\label{C28}
$x_{v^*}>\frac{1}{2}+\frac{4k-3}{8n}$.
Moreover, for any $v\in V(G)$, $x_v\ge\left(1-\frac{6t(2k+1)}{n}\right)x_{v^*}$.
\end{Claim}

\begin{proof}
By Claims \ref{C266} and \ref{C27}, we have $|W|\le t-1$. By Claim \ref{C2666}, we have $x_{v^*}\ge\frac{\rho(G)-|W|}{n-|W|}>\frac{1}{2}+\frac{4k-3}{8n}$.

Suppose that there exists $v_0\in V(G)$ such that $x_{v_0}<\left(1-\frac{6t(2k+1)}{n}\right)x_{v^*}$.
Assume that $v^*\in V_1\setminus W$.
Let $G'$ be the graph obtained from $G$ by deleting all edges incident to $v_0$ and adding all possible edges between $v_0$ and vertices in $V_2'$.
By Claim \ref{C277}, $G'$ is $F$-free.
By Claim \ref{C25}, we have
\begin{align*}
\rho(G)x_{v^*}&\le\sum_{v\in W}x_v+\sum_{v\in N_{V_1'}(v^*)}x_v+\sum_{v\in V_2'}x_v
\le|W|+t(2k+1)x_{v^*}+\sum_{v\in V_2'}x_v,
\end{align*}
which implies that $\sum_{v\in V_2'}x_v\ge\left(\rho(G)-t(2k+1)\right)x_{v^*}-|W|$.
Then, by Claim \ref{C21},
\begin{align*}
\rho(G')-\rho(G)&\ge \frac{2x_{v_0}}{\mathbf{x}^\top\mathbf{x}}\left(\sum_{v\in V_2'}x_v-\sum_{v\in N(v_0)}x_v\right)\\
&\ge \frac{2x_{v_0}}{\mathbf{x}^\top\mathbf{x}}\left(\left(\rho(G)-t(2k+1)\right)x_{v^*}-|W|-\rho(G) x_{v_0}\right)\\
&>  \frac{2x_{v_0}}{\mathbf{x}^\top\mathbf{x}}\left(\left(\rho(G)-t(2k+1)\right)x_{v^*}-t+1-\rho(G)\left(1-\frac{6t(2k+1)}{n}\right)x_{v^*}\right)\\
&=\frac{2x_{v_0}}{\mathbf{x}^\top\mathbf{x}}\left(\frac{6t(2k+1)}{n}\rho(G)x_{v^*}-t(2k+1)x_{v^*}-t+1\right)\\
&>  \frac{2x_{v_0}}{\mathbf{x}^\top\mathbf{x}}\left(\frac{6t(2k+1)(n+k+2t-\frac{7}{2})}{2n}\left(\frac{1}{2}+\frac{4k-3}{8n}\right)-t(2k+1)-t\right)\\
&> \frac{2x_{v_0}}{\mathbf{x}^\top\mathbf{x}}\left(\frac{6t(2k+1)n}{2n}\times\frac{1}{2}-2t(k+1)\right) \\
&=  \frac{2x_{v_0}}{\mathbf{x}^\top\mathbf{x}}\left(\frac{3t(2k+1)}{2}-2t(k+1)\right)\\
&>  0,
\end{align*}
a contradiction.
\end{proof}

\begin{Claim}\label{C29}
For any $v\in W$, $d_G(v)=n-1$. Moreover, $|W|=t-1$.
\end{Claim}

\begin{proof}
Suppose that there exists $v_0\in W$ such that $d_G(v_0)\le n-2$.
Let $u_0$ be a vertex that is not adjacent to $v_0$ in $G$.
Then $\rho(G+u_0v_0)>\rho(G)$, so  $G+u_0v_0$ contains $F$.
Take $H_1:=P_1\vee P_{2k}$ in $F$ with  $u_0v_0\in E(H_1)$ in $G+u_0v_0$.
Let $H_2=F- V(H_1)$. Then $|V(H_2)|=(t-1)(2k+1)$.
By Claim \ref{C26}, $G-(V(H_2)\cup W\setminus\{v_0\})$ contains  $H_3:=P_1\vee P_{2k}$.
Therefore, $G$ contains  $H_2\cup H_3$, which is $F$, a contradiction.
So for any $v\in W$, $d_G(v)=n-1$.

By Claims \ref{C266} and \ref{C27}, we have $|W|\le t-1$.
Suppose that $|W|\le t-2$.
Note that $G$ contains $H:=(t-1)(P_1\vee P_{2k})$.

If there exists $w_0\in W\setminus V(H)$, then Claim \ref{C26} implies that
  $G-((W\cup V(H))\setminus\{w_0\})$ contains $G_1:=P_1\vee P_{2k}$, so $G$
  contains $H\cup G_1$, which is $F$, a contradiction.
Thus $W\subseteq V(H)$, and moreover
 $|V(H)\setminus W|=(t-1)(2k+1)-|W|>(t-1)(2k+1)-(t-1)=2k(t-1)$.
As $V(H)\setminus W\subseteq V_1\cup V_2$, $|V_i\cap (V(H)\setminus W)|\ge k(t-1)$ for some $i=1,2$.
Let $\widehat{V_i}\subseteq V_i\cap (V(H)\setminus W)$ such that $|\widehat{V_i}|=t-1-|W|$.
Construct a new graph $G'$  from $G-E([V(H)\setminus W])$ by adding all possible edges between $\widehat{V_i}$ and $V_i\setminus(W\cup \widehat{V_i})$.
As $W\subseteq V(H)$,
there are at most $t-1$ copies $P_1\vee P_{2k}$ in $G'$, so $G'$ is $F$-free.

For any $v\in \widehat{V_i}$, by Claim \ref{C25}, $d_{G[V_i']}(v)\le(t-1)(2k+1)+k-1$, and by Claim \ref{C22}, $d_{G'[V_i']}(v)\ge|V_i|-|W|-|\widehat{V_i}|=|V_i|-(t-1)\ge\frac{n}{2}-\eta n-t+1$.
Thus,
\begin{align*}
\sum_{v\in \widehat{V_i}}\sum_{\substack{z\in V_i'\setminus\widehat{V_i} \\ vz\in E(G')}}1&\ge\sum_{v\in \widehat{V_i}}\left(\frac{n}{2}-\eta n-t+1\right)\\
&= |\widehat{V_i}|\left(\frac{n}{2}-\eta n-t+1\right)\\
&\ge (t-1-|W|)\left(\frac{n}{2}-\eta n-t+1\right)\\
&\ge \frac{n}{2}-\eta n-t+1,
\end{align*}
\begin{align*}
\sum_{v\in \widehat{V_i}}\sum_{\substack{z\in V_i'\setminus(V(H)\cap V_i)\\ vz\in E(G)}}1
&= \sum_{v\in \widehat{V_i}}d_{V_i'}(v)\le \sum_{v\in \widehat{V_i}}((t-1)(2k+1)+k-1)\\
&\le |\widehat{V_i}|((t-1)(2k+1)+k-1)\\
&= (t-1-|W|)((t-1)(2k+1)+k-1)\\
&\le (t-1)^2(2k+1)+(t-1)(k-1)
\end{align*}
and
\begin{align*}
\frac{1}{2}\sum_{v\in V(H)\setminus W} \sum_{\substack{z\in V(H)\setminus W\\ vz\in E(G)}}1
&\le {(t-1)(2k+1)-|W| \choose 2}\\
&\le {(t-1)(2k+1) \choose 2}\\
&= \frac{(t-1)^2(2k+1)^2-(t-1)(2k+1)}{2}.
\end{align*}
By Claim \ref{C28}, for any $v, z\in V(G)$, we have
\begin{align*}
x_{v}x_{z}&\ge \left(1-\frac{6t(2k+1)}{n}\right)^2x_{v^*}^2>
\left(1-\frac{6t(2k+1)}{n}\right)^2\left(\frac{1}{2}+\frac{4k-3}{8n}\right)^2\\
&= \left(\frac{1}{2}-\frac{24t(2k+1)-4k+3}{8n}-\frac{3t(2k+1)(4k-3)}{4n^2}\right)^2\\
&> \left(\frac{1}{4}-\frac{24t(2k+1)-4k+3}{8n}-\frac{3t(2k+1)(4k-3)}{4n^2}\right).
\end{align*}
Thus
\begin{align*}
&\quad \rho(G')-\rho(G)\\
&\ge\frac{2}{\mathbf{x}^\top\mathbf{x}}
\left(\sum_{v\in \widehat{V_i}}\sum_{\substack{z\in V_i'\setminus\widehat{V_i} \\ vz\in E(G')}}x_vx_z-\sum_{v\in \widehat{V_i}}\sum_{\substack{z\in V_i'\setminus(V(H)\cap V_i)\\ vz\in E(G)}}x_{v}x_{z}-\frac{1}{2}\sum_{v\in V(H)\setminus W} \sum_{\substack{z\in V(H)\setminus W\\ vz\in E(G)}}x_{v}x_{z}\right)\\
&> \frac{2}{\mathbf{x}^\top\mathbf{x}}\left(\frac{n}{2}-\eta n-t+1\right)\left(\frac{1}{4}-\frac{24t(2k+1)-4k+3}{8n}-\frac{3t(2k+1)(4k-3)}{4n^2}\right)\\
&\quad -\frac{2}{\mathbf{x}^\top\mathbf{x}}\left((t-1)^2(2k+1)+(t-1)(k-1)+\frac{(t-1)^2(2k+1)^2-(t-1)(2k+1)}{2}\right)\\
&>  0
\end{align*}
for sufficiently large $n$, a contradiction.
It follows that $|W|\ge t-1$, so $|W|=t-1$.
\end{proof}

\begin{Claim}\label{C10}
For $i=1,2$ and $v\in V_i'$,
$|V_{3-i}'\setminus N_{V_{3-i}'}(v)|\le d_{V_i'}(v)$.
\end{Claim}

\begin{proof}
Let $G'$ be the graph obtained from $G$ by deleting all edges incident to $v$ in $G[V_i']$ and adding all possible edges between $v$ and vertices in $V_{3-i}'$.
By Claim \ref{C277},  $G'$ is $F$-free.
By Claim \ref{C28},
\begin{align*}
0\le \rho(G)-\rho(G')&\le \frac{2x_v}{\mathbf{x}^\top\mathbf{x}}\left(d_{V_i'}(v)x_{v^*}-\sum_{w\in V_{3-i}'\setminus N_{V_{3-i}'}(v)}x_w\right)\\
&\le \frac{2x_vx_{v^*}}{\mathbf{x}^\top\mathbf{x}}\left(d_{V_i'}(v)-\left(1-\frac{6t(2k+1)}{n}\right)|V_{3-i}'\setminus N_{V_{3-i}'}(v)|\right),
\end{align*}
implying that $d_{V_i'}(v)-(1-\frac{6t(2k+1)}{n})|V_{3-i}'\setminus N_{V_{3-i}'}(v)|\ge0$, from which we have
\[
\left|V_{3-i}'\setminus N_{V_{3-i}'}(v)\right|\le\frac{d_{V_i'}(v)}{1-\frac{6t(2k+1)}{n}}
=d_{V_i'}(v)+\frac{6d_{V_i'}(v)t(2k+1)}{n-6t(2k+1)}<d_{V_i'}(v)+1
\]
as $d_{V_i'}(v)\le(t-1)(2k+1)+k-1$ by Claim \ref{C25}.
Thus, $|V_{3-i}'\setminus N_{V_{3-i}'}(v)|\le d_{V_i'}(v)$.
\end{proof}

\begin{Claim}\label{DD}
For any $v\in V_i'$, $d_{V_i'}(v)\le k-1$ for $i=1,2$.
\end{Claim}

\begin{proof}
Recall that $G$ contains $H:=(t-1)(P_1\vee P_{2k})$.
By Claim \ref{C29}, $W\subset V(H)$.
Then we show that every $P_1\vee P_{2k}$ in $H$ contains at almost one vertex of $W$.
Otherwise,  $H$ contains $G_1:=P_1\vee P_{2k}$ with $|V(G_1)\cap W|\ge 2$.
Let $\{v_1, v_2\}\subseteq V(G_1)\cap W$ and $H'=H-V(G_1)$.
By Claim \ref{C29}, we have $d_G(v_i)=n-1$ for any $i=1,2$. Then, by Claim \ref{C26},
we can find  $H'':=2(P_1\vee P_{2k})$ in $G$ such that $V(H'')\cap V(H')=\emptyset$.
So $G$ contains  $H'\cup H''$, which is $F$, a contradiction.
So, indeed, by Claim \ref{C29}, every $P_1\vee P_{2k}$ in $H$ contains exactly one vertex of $W$.

Next, we show that $d_{V_i'}(v)\le k-1$  for any $v\in V_i'$.
Suppose that this is not true. Then $d_{V_i'}(v_0)\ge k$ for some $v_0\in V_i'$.
For any $v\in V_i'$, we have $d_G(v)>(\frac{1}{2}-2\eta)n$ as $v\notin U$ and $d_{V_i}(v)< 2\eta n$ as $v\notin W$.
Then
\[
d_{V_{3-i}}(v)=d_G(v)-d_{V_i}(v)>\left(\frac{1}{2}-2\eta\right)n-2\eta n=\frac{n}{2}-4\eta n.
\]
By Claim \ref{C29},
\[
d_{V_{3-i}'}(v)\ge d_{V_{3-i}}(v)-|W|>\frac{n}{2}-4\eta n-(t-1)\ge\frac{n}{2}-5\eta n.
\]
Assume that $\{v_1, \dots, v_k\}\subseteq N_{V_i'}(v_0)$. 
By Lemma \ref{cap} and Claim \ref{C22}, we have
\begin{align*}
|N_{V_{3-i}'}(v_0)\cap\dots\cap N_{V_{3-i}'}(v_k)| &\ge |N_{V_{3-i}'}(v_0)|+\dots+|N_{V_{3-i}'}(v_k)|-k|V_{3-i}|\\
&>(k+1)\left(\frac{n}{2}-5\eta n\right)-k\left(\frac{n}{2}+\eta n\right)\\
&= \frac{n}{2}-(6k+5)\eta n\\
&=\left(\frac{1}{2}-\frac{6k+5}{9(3k+1)}\right)n\\
&> k
\end{align*}
as $n$ is sufficiently large.
Therefore, there has a copy of $H':=P_1\vee P_{2k}$ with center $v_0$ in $G-W$.
By Claims \ref{C26} and \ref{C29}, we can find $H'':=(t-1)P_1\vee P_{2k}$ with center in $W$ such that $V(H')\cap V(H'')=\emptyset$, $G$ contains $H'\cup H''=F$, a contradiction.
\end{proof}

For $i=1,2$, let $d_i=\Delta(G[V_i'])$. Assume that $d_1\ge d_2$.  Then $d_1\le k-1$ by Claim \ref{DD}.
By Claim \ref{C29}, $|V_1'|+|V_2'|=n-t+1$. Assume that $|V_1'|=\frac{n-t+1}{2}+r$ and $|V_2'|=\frac{n-t+1}{2}-r$, where $r$ is not necessarily nonnegative.
Let
\[
B_r=\begin{pmatrix}
t-2 & \frac{n-t+1}{2}+r & \frac{n-t+1}{2}-r \\
t-1 & d_1 & \frac{n-t+1}{2}-r \\
t-1 & \frac{n-t+1}{2}+r & d_2\end{pmatrix}.
\]
View $\rho(B_r)$ as a function of $r$. It is easy to see that  $\rho(B_r)\ge\min\{\frac{n}{2}+\frac{t-1}{2}+r+d_2, \frac{n}{2}+\frac{t-1}{2}-r+d_1\}>\frac{n}{3}$.
By a direct calculation, the characteristic polynomial of $B_r$ is
\begin{align*}
f_r(x)&= x^3-(d_1+d_2+t-2)x^2\\
&\quad -\left((t-1)(n-t+1)-(d_1+d_2)(t-2)+\frac{1}{4}(n-t+1)^2-d_1d_2-r^2\right)x\\
&\quad +\frac{1}{2}(d_1+d_2)(t-1)(n-t+1)-r(d_1-d_2)(t-1)-d_1d_2(t-2)\\
&\quad -\frac{1}{4}t(t-1)^2-\frac{1}{4}nt(n-2t+2)+r^2t.
\end{align*}
Then
\begin{align*}
f'_r(x)&=3x^2-2(d_1+d_2+t-2)x\\
&\quad -\left((t-1)(n-t+1)-(d_1+d_2)(t-2)+\frac{1}{4}(n-t+1)^2-d_1d_2-r^2\right),
\end{align*}
it can be checked that its largest root is less than $\frac{n}{3}$, implying that $f_r(x)$ is increasing for $x\ge\frac{n}{3}$.
For $s>0$,
\[
f_s(x)-f_{-s}(x)=-2s(d_1-d_2)(t-1)<0,
\]
from which we have $\rho(B_\frac{1}{2})>\rho(B_{-\frac{1}{2}})$, $\rho(B_1)>\rho(B_{-1})$ and $\rho(B_r)>\rho(B_{-r})$ if $r>1$.

For $a\le 0$, if   $x>\frac{n}{3}$, then
\begin{align*}
f_a(x)-f_{\frac{1}{2}-a}(x)&=-\frac{4a-1}{4}(-x-t+2(d_1-d_2)(t-1))< 0
\end{align*}
as $x$ is sufficiently large. So
$\rho(B_0)>\rho(B_{\frac{1}{2}})$, $\rho(B_{-\frac{1}{2}})>\rho(B_1)$ and $\rho(B_{r})>\rho(B_{\frac{1}{2}-r})$ if $r\le -1$.
Therefore,
\begin{align}\label{3rd}
\rho(B_0)>\rho(B_\frac{1}{2})>\rho(B_{-\frac{1}{2}})>\rho(B_1)>\rho(B_{-1})>\rho(B_{|r|})
\end{align}
for $|r|\ge1$.

Let $G^*=K_{t-1}\vee(G_1^*\vee(\frac{n-t+1}{2}-r)K_1)$, where $G_1^*$ is $P_{2k}$-free and is $(k-1)$-regular if $(k-1)|V(G_1^*)|$ is even and nearly $(k-1)$-regular otherwise.
It is easy to see that $G^*$ is $F$-free, so $\rho(G)\ge\rho(G^*)$.

Let $I_n(i)$ be the $n\times n$ diagonal matrix with only nonzero entry $1$ in its $(i,i)$-entry.

\begin{Claim}\label{C11}
Let $c$ be a constant and $G'$ be a graph obtained from $G$ by deleting at most $c$ edges from $G-W$, then $\Delta(G'[V_1'])+\Delta(G'[V_2'])\ge k-1$.
\end{Claim}

\begin{proof}
Let $\Delta(G'[V_1'])=d_1'$ and $\Delta(G'[V_2'])=d_2'$.
Let
\[
Q_1=\begin{pmatrix}
t-2 & |V_1'| & |V_2'| \\
t-1 & d_1' & |V_2'| \\
t-1 & |V_1'| & d_2'\end{pmatrix}.
\]
By Lemmas  \ref{none} and \ref{QM} and
$\rho(G')\le \rho(Q_1)\le \rho(B_r)$.

Suppose that $d_1'+d_2'\le k-2$.
By Claim \ref{C28}, together with \eqref{3rd}, we have
\begin{align*}
\rho(G)&=\frac{\mathbf{x}^\top A(G)\mathbf{x}}{\mathbf{x}^\top\mathbf{x}}
\le \frac{\mathbf{x}^\top A(G')\mathbf{x}}{\mathbf{x}^\top\mathbf{x}}+\frac{2cx_{v^*}^2}{\mathbf{x}^\top\mathbf{x}}\\
&\le \rho(G')+\frac{2cx_{v^*}^2}{\mathbf{x}^\top\mathbf{x}}\\
&\le \rho(B_r)+\frac{2c}{n(1-\frac{6t(2k+1)}{n})^2}\\
&\le \rho(B_0)+\frac{2c}{n(1-\frac{6t(2k+1)}{n})^2}.
\end{align*}
Let $r=0,\frac{1}{2}$ in $G^*$. Then
$G^*=K_{t-1}\vee(G_1^*\vee(\lfloor\frac{n-t+1}{2}\rfloor)K_1)$, where $|V(G_1^*)|=\lceil\frac{n-t+1}{2}\rceil$. Let
\[
A=\begin{pmatrix}
t-2 & \lceil\frac{n-t+1}{2}\rceil & \lfloor\frac{n-t+1}{2}\rfloor \\
t-1 & k-1 & \lfloor\frac{n-t+1}{2}\rfloor \\
t-1 & \lceil\frac{n-t+1}{2}\rceil & 0\end{pmatrix}.
\]
Denote by $(1, y_1,y_2)^\top$ be a positive eigenvector associated with $\rho(A)$ and
\[
\mathbf{y}=(\underbrace{1, \dots, 1}_{t-1}, \underbrace{y_1, \dots, y_1}_{\lceil\frac{n-t+1}{2}\rceil}, \underbrace{y_2, \dots, y_2}_{\lfloor\frac{n-t+1}{2}\rfloor})^\top.
\]
Since
$\rho(A)=t-2+\lceil\frac{n-t+1}{2}\rceil y_1+\lfloor\frac{n-t+1}{2}\rfloor y_2$ and
$\rho(A)y_1=t-1+(k-1)y_1+\lfloor\frac{n-t+1}{2}\rfloor y_2$,
we have
\[
\rho(A)y_1=\rho(A)+1+\left(k-1-\left\lceil\frac{n-t+1}{2}\right\rceil\right)y_1,
\]
so $y_1=\frac{\rho(A)+1}{\rho(A)+\lceil\frac{n-t+1}{2}\rceil-k+1}$.
Similarly, 
$y_2=\frac{\rho(A)+1}{\rho(A)+\lfloor\frac{n-t+1}{2}\rfloor}$.
It follows that $1>y_1>y_2$.

If $G_1^*$ is $(k-1)$-regular, then $A$ is  the quotient matrix of $A(G^*)$ corresponding to the equitable partition $V(G^*)=V(K_{t-1})\cup V(G_1^*)\cup V((\lfloor\frac{n-t+1}{2}\rfloor)K_1)$, so we have
by Lemma \ref{QM}, $\rho(G)\ge\rho(G^*)=\rho(A)$.
Suppose that $G_1^*$ is nearly $(k-1)$-regular. Assume that $v\in V(G_1^*)$ such that $d_{V(G_1^*)}(v)=k-2$, so $A$ is  the quotient matrix of $A(G^*)+I_n(v)$ corresponding to the equitable partition $V(G^*)=V(K_{t-1})\cup V(G_1^*)\cup V((\lfloor\frac{n-t+1}{2}\rfloor)K_1)$.
By Lemma \ref{QM}, $\mathbf{y}$ is an eigenvector associated with $\rho(A(G^*)+I_n(v))$. So
\begin{align*}
\rho(G)\ge\rho(G^*)&\ge \frac{\mathbf{y}^\top A(G^*)\mathbf{y}}{\mathbf{y}^\top\mathbf{y}}
=\frac{\mathbf{y}^\top(A(G^*)+I_n(v))\mathbf{y}}{\mathbf{y}^\top\mathbf{y}}-
\frac{\mathbf{y}^\top I_n(v)\mathbf{y}}{\mathbf{y}^\top\mathbf{y}}>\rho(A)-\frac{1}{ny_2^2}.
\end{align*}
Thus, $\rho(G)>\rho(A)-\frac{1}{ny_2^2}$ whether $G_1^*$ is $(k-1)$-regular or
$G_1^*$ is nearly $(k-1)$-regular.
By a direct calculation, the characteristic polynomial of $A$ is
\begin{align*}
h(x)&= x^3-(k+t-3)x^2\\
&\quad -\left((n-t+1)(t-1)+\left\lfloor\frac{n-t+1}{2}\right\rfloor\left\lceil\frac{n-t+1}{2}\right\rceil-(k-1)(t-2)\right)x\\
&\quad -t\left\lfloor\frac{n-t+1}{2}\right\rfloor\left\lceil\frac{n-t+1}{2}\right\rceil+(k-1)(t-1)\left\lfloor\frac{n-t+1}{2}\right\rfloor.
\end{align*}
For $x>\frac{n}{2}$, since $d_1+d_2\le k-2$, we have from the expressions for $f_0(x)$ and $h(x)$ that
\begin{align*}
&\quad f_0(x)-h\left(x+\frac{1}{ny_2^2}+\frac{2c}{n(1-\frac{6t(2k+1)}{n})^2}\right)\\
&= x^3-(d_1+d_2+t-2)x^2-\left(x+\frac{1}{ny_2^2}+\frac{2c}{n(1-\frac{6t(2k+1)}{n})^2}\right)^3\\
&\quad +(k+t-3)\left(x+\frac{1}{ny_2^2}+\frac{2c}{n(1-\frac{6t(2k+1)}{n})^2}\right)^2+O(n)\\
&\ge x^3-(k+t-4)x^2-\left(x^3+3\left(\frac{1}{ny_2^2}+\frac{2c}{n(1-\frac{6t(2k+1)}{n})^2}\right)x^2\right)\\
&\quad +(k+t-3)x^2+O(n)\\
&\ge\left(1-\frac{3}{ny_2^2}-\frac{6c}{n-12t(2k+1)}\right)x^2+O(n)\\
&>\left(1-\frac{3}{ny_2^2}-\frac{6c}{n-12t(2k+1)}\right)\frac{n^2}{4}+O(n)>0,
\end{align*}
where the last inequality follows as  $\rho (A)$ is bounded from below by the minimum row sum of $A$, which implies that $y_2=\frac{\rho(A)+1}{\rho(A)+\lfloor\frac{n-t+1}{2}\rfloor}>\frac{1}{2}$. 
So
\[
\rho(B_0)<\rho(A)-\frac{1}{ny_2^2}-\frac{2c}{n(1-\frac{6t(2k+1)}{n})^2}.
\]
It now follows that $\rho(G)\le\rho(B_0)+\frac{2c}{n(1-\frac{6t(2k+1)}{n})^2}<\rho(A)-\frac{1}{ny_2^2}$, a contradiction.
\end{proof}

\begin{Claim}\label{C12}
$\Delta(G[V_1'])=k-1$ and $E(G[V_2'])=\emptyset$.
\end{Claim}

\begin{proof}
By Claim \ref{DD},
we have $\Delta(G[V_i'])\le k-1$ for  $i=1,2$.
It suffices to show  that $\Delta(G[V_i'])\ge k-1$.
As  $\Delta(G[V_1'])\ge\Delta(G[V_2'])$, we have $\Delta(G[V_1'])\ge\frac{k-1}{2}$
by Claim \ref{C11} with $c=0$.

Suppose that $k=3$ and $\Delta(G[V_1'])=1$. By Claim \ref{C11}, we have $\Delta(G[V_2'])=\Delta(G[V_1'])=1$.
%
Let $B'$ and $B''$ be the matrix of $B_r$ with $d_1=d_2=1$ and $d_1=2$, $d_2=0$, respectively.
Denote by $g_r(x)$ and $h_r(x)$  the characteristic polynomial of $B'$ and $B''$, respectively.
By a direct calculation, $g_r(x)-h_r(x)=x-(t-1)(2r-1)+1>0$ for $x\ge\rho(B'')$, implying that $\rho(B')<\rho(B'')$.
Since $(k-1)|V_1'|=2|V_1'|$ is even, $G_1^*$ is $2$-regular  and $\rho(G^*)=\rho(B'')$ by Lemma \ref{QM}.
If $G[V_i']$ is $1$-regular for any $i=1,2$, then by Lemma \ref{QM}, $\rho(G)=\rho(B')<\rho(B'')=\rho(G^*)$, a contradiction.
So $G[V_i']$ is nearly $1$-regular for some $i=1,2$. By Lemma \ref{jn}, $\rho(G)\le\rho(B')<\rho(B'')=\rho(G^*)$, also a contradiction.
It now follows that  $\Delta(G[V_1'])\ge 2=k-1$, as desired.

Suppose that $k\ge4$.
Recall that $G$ contains $(t-1)$ copies of $P_1\vee P_{2k}$, then the center of each is in $W$ by Claims \ref{C29} and \ref{DD},
so $G[V_1'\cup V_2']$ is $(P_1\vee P_{2k})$-free.

Firstly, we  show that $\Delta(G[V_1'])\ge k-2$.
Assume that $v\in V_1'$ with $d_{V_1'}(v)=\Delta(G[V_1'])\ge\frac{k-1}{2}$.
Let $\left\{v_1,\dots,v_{\lceil\frac{k-1}{2}\rceil}\right\}\subseteq N_{V_1'}(v)$.
Let
$V_2''=\bigcap_{i=0}^{\lceil\frac{k-1}{2}\rceil}N_{V_2'}(v_i)$ with $v_0=v$.
We assert that $G[V_2'']$ contains no  $\lceil\frac{k}{2}\rceil P_3$. Otherwise,
there is a path $P$ containing all the vertices of some $\lceil\frac{k}{2}\rceil P_3$ and
in $N_{V_1'}(v)$, so $|V(P)|\ge 2k$ and
that $G[\{v\}\cup V(P)]$ contains $P_1\vee P_{2k}$, contradicting the fact that $G[V_1'\cup V_2']$ is $(P_1\vee P_{2k})$-free.
For any $u_1, u_2\in V(H)$, if the distance between $u_1$ and $u_2$ is at least $3$, $d_H(u_1)\ge2$ and  $d_H(u_2)\ge2$, then they belong to two vertex-disjoint $P_3$.
Note that for  any vertex in $V_2''$, there are at most $(k-1)(k-2)+k-1+1=(k-1)^2+1$ vertices  in $V_2''$ within distance $2$ from it.
So there are less than $\lceil\frac{k}{2}\rceil\left((k-1)^2+1\right)<\frac{k^3}{2}$ vertices with degree at least $2$ in $G[V_2'']$.
Then
\[
e(G[V_2''])=\frac{1}{2}\sum\limits_{w\in V_2''}d_{V_2''}(w)<\frac{1}{2}\left(|V_2''|-\frac{k^3}{2}+\frac{k^3}{2}(k-1)\right)=\frac{|V_2''|}{2}+\frac{k^3(k-2)}{4}.
\]
By Claim \ref{C10}, we have
\[
|V_2'\setminus V_2''|\le \left(\left\lceil\frac{k-1}{2}\right\rceil+1\right)(k-1)=\left\lceil\frac{k+1}{2}\right\rceil(k-1).
\]
So
\[
e(G[V_2'])\le e(G[V_2''])+|V_2'\setminus V_2''|(k-1)=e(G[V_2''])+|V_2'\setminus V_2''|(k-1)<\frac{|V_2''|}{2}+c_0,
\]
where
\[
c_0=\frac{k^3(k-2)}{4}+\left\lceil\frac{k+1}{2}\right\rceil(k-1)^2.
\]
Let $G'$ be the graph obtained from $G$ be deleting at most $c_0$ edges in $G[V_2']$ such that $\Delta(G'[V_2'])\le 1$.
By Claim \ref{C11}, we have $\Delta(G'[V_1'])+\Delta(G'[V_2'])\ge k-1$.
So  $\Delta(G[V_1'])=\Delta(G'[V_1'])\ge k-2$.

Secondly,
we show that $\Delta(G[V_1'])\ge k-1$.
Let $\{v_1,\dots,v_{k-2}\}\subseteq N_{V_1'}(v)$ and $V_2''=\bigcap_{i=0}^{k-2}N_{V_2'}(v_i)$ with $v_0=v$.
We assert that  $\nu(G[V_2''])\le2$. Otherwise,   $\nu(G[V_2''])\ge3$. By Claims \ref{C22} and \ref{C29}, $|V_2''|\ge|V_2|-|W|-|V_2'\setminus V_2''|>\frac{n}{2}-\eta n-(t-1)-(k-1)^2$, so we can require that $|V_2''|\ge |V(3K_2)|+k-4$.  Consequently,
there is a path containing the vertices of $3K_2$, $k-4$ vertices in $V_2''\setminus V(3K_2)$ and the vertices from $N_{V_1'}(v)$ with order is at least $6+k-4+k-2=2k$, implying that
the vertices of this path together with $v$ will generate a copy of $P_1\vee P_{2k}$, a contradiction.
So by Lemma \ref{MD}, we have
\[
e(G[V_2''])\le\nu(G[V_2''])(\Delta(G[V_2''])+1)\le2k.
\]
By  Claim \ref{C10},  $|V_2'\setminus V_2''|\le(k-1)^2$.
Thus,
\[
e(G[V_2'])\le e(G[V_2''])+|V_2'\setminus V_2''|(k-1)\le2k+(k-1)^3.
\]
Let $G''=G-E(G[V_2'])$.
Then $\Delta(G''[V_2'])=0$.
By Claim \ref{C11}, we have $\Delta(G[V_1'])=\Delta(G''[V_1'])\ge k-1$.

Let $V_1^*=\{v\in V_1': d_{V_1'}(v)=k-1\}$.
Suppose  that $|V_1^*|$ is independent of $n$.
Let $\widehat G$ be a graph obtained from $G$ by deleting $E(G[V_2'])$ and $|V_1^*|$ edges in $G[V_1']$ such that $\Delta(\widehat G[V_1'])=k-2$.
So $\Delta(\widehat G[V_1'])+\Delta(\widehat G[V_2'])<k-1$, contradicting Claim \ref{C11}.
So $|V_1^*|$ is large enough.
Suppose that there exists $v_1v_2\in E(G[V_2'])$.
Let $V_1''=N_{V_1'}(v_1)\cap N_{V_1'}(v_2)$.
By Claim \ref{C10}, $|V_1'\setminus V_1''|\le 2(k-1)$.
Suppose that for any $w\in V_1^*$, there is $|(N_{V_1'}(w)\cup\{w\})\cap(V_1'\setminus V_1'')|\ge1$.
Then $|V_1^*|\le|V_1'\setminus V_1''|(k-1)\le 2(k-1)^2$, a contradiction.
So there exists $v_0\in V_1^*\cap V_1''$ such that $N_{V_1'\setminus V_1''}(v_0)=\emptyset$, i.e.,
$d_{V_1''}(v_0)=k-1$.
Let $N_{V_1''}(v_0)=\{v_1', \dots, v_{k-1}'\}$.
Then by Claim \ref{C10},
\[
|V_2'|-|N_{V_2'}(v_0)\cap N_{V_2'}(v_1')\cap\dots\cap N_{V_2'}(v_{k-1}')|\le k(k-1).
\]
So $G[N_{V_2'}(v_0)\cap N_{V_2'}(v_1')\cap\dots\cap N_{V_2'}(v_{k-1}')]$ contains  $K_2\cup (k-1)K_1$, implying that $G[V_1'\cup V_2']$ contains $P_1\vee P_{2k}$, a contradiction.
\end{proof}

By Claims \ref{C10} and \ref{C12}, we have $G=K_{t-1}\vee(G_1\vee (\frac{n-t+1}{2}-r)K_1)$, where $G_1=G[V_1']$ with $\Delta(G_1)=k-1$.

Let
\[
A_r=\begin{pmatrix}
t-2 & \frac{n-t+1}{2}+r & \frac{n-t+1}{2}-r \\
t-1 & k-1 & \frac{n-t+1}{2}-r \\
t-1 & \frac{n-t+1}{2}+r & 0\end{pmatrix}.
\]
Let $(1, z_1,z_2)^\top$ be a positive eigenvector associated with $\rho(A_r)$. It can be checked that $1>z_1>z_2$ by similar argument as in the proof of Claim \ref{C11}.
Let
\[
\mathbf{z}=(\underbrace{1, \dots, 1}_{t-1}, \underbrace{z_1, \dots, z_1}_{|V(G_1^*)|}, \underbrace{z_2, \dots, z_2}_{\frac{n-t+1}{2}-r})^\top.
\]

If $G_1^*$ is $(k-1)$-regular, then
$A_r$ is the quotient matrix of $G^*$ corresponding to the partition $V(G^*)=V(K_{t-1})\cup  V(G_1^*)\cup V((\frac{n-t+1}{2}-r)K_1)$.
By Lemma \ref{QM}, we have
\[
\rho(G)\ge\rho(G^*)=\rho(A_r).
\]

Suppose that $G_1^*$ is nearly $(k-1)$-regular. Assume that $v\in V(G_1^*)$ such that  $d_{G_1^*}(v)=k-2$, so $A_r$ is  the quotient matrix of $A(G^*)+I_n(v)$ corresponding to the equitable partition $V(G^*)=V(K_{t-1})\cup  V(G_1^*)\cup V((\frac{n-t+1}{2}-r)K_1)$.
By Lemma \ref{QM}, $\rho(A(G^*)+I_n(v))=\rho(A_r)$ and $\mathbf{z}$ is an eigenvector associated with $\rho(A(G^*)+I_{n}(v))$.
Hence,
\begin{align*}
\mathbf{x}^\top (\rho(G)-\rho(A_r))\mathbf{z}&\ge \mathbf{x}^\top (\rho(G^*)-\rho(A(G^*)+I_n(v)))\mathbf{z}\\
&\ge \mathbf{x}^\top (A(G^*)-(A(G^*)+I_n(v)))\mathbf{z}\\
&= -x_vz_1\\
&\ge -x_{v^*}z_1.
\end{align*}
Thus, in either case, we have
\begin{equation}\label{E1}
\mathbf{x}^\top (\rho(G)-\rho(A_r))\mathbf{z}\ge -x_{v^*}z_1.
\end{equation}

\begin{Claim}\label{C13}
$G=G^*$ for some $r$.
\end{Claim}

\begin{proof}
Suppose to the contrary that $G_1$ is not $(k-1)$-regular if $(k-1)|V_1'|$ is even and $G_1$ is not nearly $(k-1)$-regular if $(k-1)|V_1'|$ is odd.
%
%
Let $I'$ be a diagonal matrix with $n\times n$ such that the quotient matrix of $A(G)+I'$ according to the equitable partition $V(G)=V(K_{t-1})\cup V_1'\cup V_2'$ is $A_r$.
Note that $2e(G_1)=\sum\limits_{v\in V_1'}d_{V_1'}(v)$.
Let $\sum\limits_{v\in V_1'}d_{V_1'}(v)=(k-1)|V_1'|-a$, which is even, then $a\ge2$ if $(k-1)|V_1'|$ is even and $a\ge3$ otherwise.
So the sum of non-zero entries in $I'$ is at least $2$.
Therefore, by Lemma \ref{QM} and Claim \ref{C28}, we have
\begin{align*}
\mathbf{x}^\top(\rho(G)-\rho(A_r))\mathbf{z}&= \mathbf{x}^\top(\rho(G)-\rho(A(G)+I'))\mathbf{z}=\mathbf{x}^\top(A(G)-(A(G)+I'))\mathbf{z}\\
&=-\mathbf{x}^\top I'\mathbf{z}=-\sum\limits_{i=1}^nI'_{v_i,v_i}x_{v_i}z_{v_i}=-\sum\limits_{i=1}^nI'_{v_i,v_i}x_{v_i}z_1\\
&\le-\sum\limits_{i=1}^nI'_{v_i,v_i}\left(1-\frac{6t(2k+1)}{n}\right)x_{v^*}z_1\\
&\le -2\left(1-\frac{6t(2k+1)}{n}\right)x_{v^*}z_1\\
&<-x_{v^*}z_1,
\end{align*}
contradicting (\ref{E1}).
\end{proof}

By Claim \ref{C13},
$G=K_{t-1}\vee(G_1^*\vee(\frac{n-t+1}{2}-r)K_1)$ for some $r$.

\begin{Claim} \label{Fca}
If $k$ is odd, then
\[
r=\begin{cases}
0 & \mbox{if $n-t$ is odd},\\
\frac{1}{2} & \mbox{otherwise}
\end{cases}
\]
and if $k$ is even, then
\[
r=\begin{cases}
0 & \mbox{if $n-t\equiv 3\!\!\pmod{4}$},\\
-\frac{1}{2} & \mbox{if  $n-t\equiv 0\!\!\pmod{4}$},\\
1 & \mbox{if $n-t\equiv 1\!\!\pmod{4}$},\\
\frac{1}{2} & \mbox{if $n-t\equiv 2\!\!\pmod{4}$}.
\end{cases}
\]
\end{Claim}

\begin{proof} We write $G$ as $G(r)$ for fixed $n,k$ and $t$.
In the following,  we will determine the value of $r$.
From (\ref{3rd}), we have
\begin{align}\label{E2}
\rho(A_0)>\rho\left(A_{\frac{1}{2}}\right)>\rho\left(A_{-\frac{1}{2}}\right)>\rho(A_1)
>\rho(A_{-1})>\rho(A_{|r|})
\end{align}
with $|r|>1$.

\noindent{\bf Case 1.} $k$ is odd.

By Claim \ref{C13}, we have $G_1^*$ is $(k-1)$-regular.
Then $\rho(G(r))=\rho(A_r)$ by Lemma \ref{QM}. So by (\ref{E2}), we have $r=0$ if $n-t$ is odd, and $r=\frac{1}{2}$ otherwise.

\noindent{\bf Case 2.} $k$ is even.

\noindent{\bf Case 2.1.} $n-t\equiv 3\!\!\pmod{4}$.

If $(k-1)(\frac{n-t+1}{2}+r)$ is even, then $G_1^*$ is $(k-1)$-regular by Claim \ref{C13}, and
by Lemma \ref{QM}, we have $\rho(G(r))=\rho(A_r)$, so $r=0$ from (\ref{E2}).
Suppose that $(k-1)(\frac{n-t+1}{2}+r)$ is odd. Then $|r|\ge1$ and $G_1^*$ is nearly $(k-1)$-regular by Claim \ref{C13}, and by Lemma \ref{jn} and (\ref{E2}), we have $\rho(G(r))\le\rho(A_{r})<\rho(A_0)$, contradicting the choice of $G$.
So $(k-1)(\frac{n-t+1}{2}+r)$ is even, and then $r=0$.

\noindent{\bf Case 2.2.} $n-t\equiv 0\!\!\pmod{4}$.

If $(k-1)(\frac{n-t+1}{2}+r)$ is even, then $\rho(G(r))=\rho\left(A_r\right)$ by Lemma \ref{QM} and Claim \ref{C13}, so $r=-\frac{1}{2}$ from (\ref{E2}).
Suppose that $(k-1)(\frac{n-t+1}{2}+r)$ is odd. Then $r=\frac{1}{2}$ or $|r|>\frac{1}{2}$, and $G_1^*$ is nearly $(k-1)$-regular by Claim \ref{C13}, and by Lemma \ref{jn} and (\ref{E2}), we have $\rho(G(r))\le\rho(A_{r})\le\rho\left(A_{\frac{1}{2}}\right)$.
Assume that $d_{V_1'}(v)=k-2$ for $v\in V_1'$.
Let
\[
c'=c'_{n,k,t}=\frac{(1-\frac{6t(2k+1)}{n})^2(\frac{1}{2}+\frac{k-2}{2n})^2}{n}=\frac{(\frac{1}{2}-\frac{6t(2k+1)-k+2}{2n}-\frac{6t(2k+1)(k-2)}{2n^2})^2}{n}=\frac{1}{4n}-o(n^2).
\]
By Lemma \ref{QM} and Claim \ref{C28},
\begin{align*}
\rho(G)&=\frac{\mathbf{x}^\top A(G)\mathbf{x}}{\mathbf{x}^\top \mathbf{x}}=\frac{\mathbf{x}^\top (A(G)+I_n(v)-I_n(v))\mathbf{x}}{\mathbf{x}^\top \mathbf{x}}\\
&\le \rho(A(G)+I_n(v))-\frac{\mathbf{x}^\top I_n(v)\mathbf{x}}{\mathbf{x}^\top \mathbf{x}}\\
&< \rho (A_{\frac{1}{2}})-\frac{x_v^2}{n}\\
&\le \rho (A_{\frac{1}{2}})-c'.
\end{align*}
Let $f(x)$ and $h(x)$ be the  characteristic polynomial of $A_{-\frac{1}{2}}$ and $A_{\frac{1}{2}}$, respectively.
Then
\begin{align*}
f(x)&=x^3-(k+t-3)x^2\\
&\quad -\left((t-1)(n-t+1)-(k-1)(t-2)+\frac{1}{4}(n-t)^2+\frac{n-t}{2}\right)x\\
&\quad -\frac{1}{4}(n-t+2)(tn-(t-1)(2k+t-1)+1)
\end{align*}
and
\begin{align*}
h(x)&=x^3-(k+t-3)x^2\\
&\quad -\left((t-1)(n-t+1)-(k-1)(t-2)+\frac{1}{4}(n-t)^2+\frac{n-t}{2}\right)x\\
&\quad -\frac{1}{4}(n-t)(t(n+2)-(t-1)(2k+t-1)-1).
\end{align*}
Note that for $x>\frac{n}{2}$,
\[
-3c'x^2+\frac{(n-t)^2}{4}c'<-\frac{3}{16}n+\frac{1}{16}n+O(1)=-\frac{n}{8}+O(1).
\]
By a direct calculation, for $x>\frac{n}{2}$,
\[
f(x)-h\left(x+c'\right)= -3c'x^2+\frac{(n-t)^2}{4}c'-\frac{n}{2}+O(1)<-\frac{n}{8}-\frac{n}{2}+O(1)<0.
\]
So \[
f\left(\rho(A_{\frac{1}{2}})-c'\right)
-h(\rho(A_{\frac{1}{2}}))=f(\rho(A_{\frac{1}{2}})-c')<0,
\]
implying that
\[
\rho(A_{-\frac{1}{2}})>\rho(A_{\frac{1}{2}})-c'\ge\rho(G),
\]
a contradiction. Thus, $r=-\frac{1}{2}$.


\noindent{\bf Case 2.3.} $n-t\equiv 1\!\!\pmod{4}$.

If $(k-1)(\frac{n-t+1}{2}+r)$ is even, then $\rho(G(r))=\rho\left(A_r\right)$ by Lemma \ref{QM} and Claim \ref{C13}, so $r=1$ from (\ref{E2}).
Suppose that $(k-1)(\frac{n-t+1}{2}+r)$ is odd. Then $r=0$ or $|r|\ge2$ and $G_1^*$ is nearly $(k-1)$-regular by Claim \ref{C13}.
If $|r|\ge2$, then by Lemma \ref{jn} and (\ref{E2}), we have $\rho(G(r))\le\rho(A_r)<\rho(A_1)$, a contradiction.
So $r=0$.
Assume that $d_{V_1'}(v)=k-2$ for $v\in V_1'$.
From Case 2.2., we have $c'=c'_{n,k,t}=\frac{1}{4n}-o(n^2)$.
By Lemma \ref{QM} and Claim \ref{C28},
\begin{align*}
\rho(G)&=\frac{\mathbf{x}^\top A(G)\mathbf{x}}{\mathbf{x}^\top \mathbf{x}}=\frac{\mathbf{x}^\top (A(G)+I_n(v)-I_n(v))\mathbf{x}}{\mathbf{x}^\top \mathbf{x}}<\rho (A_0)-c'.
\end{align*}
Let $\psi(x)$ and $\sigma(x)$ be the  characteristic polynomial of $A_1$ and $A_0$, respectively.
Then
\begin{align*}
\psi(x)&=x^3-(k+t-3)x^2-\frac{1}{4}\left((n+t-1)^2-4k(t-2)-(2t-3)^2-7\right)x\\
&\quad -\frac{n-t-1}{4}\left(tn-(t-1)(2k+t-2)+2t\right)
\end{align*}
and
\begin{align*}
\sigma(x)&=x^3-(k+t-3)x^2-\frac{1}{4}\left((n+t-1)^2-4k(t-2)-(2t-3)^2-3\right)x\\
&\quad -\frac{n-t-1}{4}\left(tn-(t-1)(2k+t-2)\right).
\end{align*}
Note that for $x>\frac{n}{2}$, $-3c'x^2+\frac{n^2}{4}c'<-\frac{1}{8}n+O(1)$.
By a direct calculation, for $\frac{n}{2}<x\le\Delta(G)=n-1$,
\begin{align*}
\psi(x)-\sigma\left(x+c'\right)&= -3c'x^2+x+\frac{n^2}{4}c'-\frac{t}{2}n-O(1)\\
&<-\frac{1}{8}n+n-1-\frac{t}{2}n+O(1)\\
&=\left(1-\frac{1}{8}-\frac{t}{2}\right)n+O(1)\\
&<0
\end{align*}
as $t\ge2$ and $n$ is sufficiently large.
So \[
\psi \left(\rho(A_0)-c'\right)
-\sigma(\rho(A_0))=\psi(\rho(A_0)-c')<0,
\]
implying that
\[
\rho(A_1)>\rho(A_0)-c'>\rho(G),
\]
a contradiction. It follows that $r=1$.

\noindent{\bf Case 2.4.} $n-t\equiv 2\!\!\pmod{4}$.

If $(k-1)(\frac{n-t+1}{2}+r)$ is even, then $\rho(G(r))=\rho\left(A_r\right)$ by Lemma \ref{QM} and Claim \ref{C13}, so
$r=\frac{1}{2}$ from (\ref{E2}).
Suppose that $(k-1)(\frac{n-t+1}{2}+r)$ is odd, then $r=-\frac{1}{2}$ or $|r|\ge\frac{3}{2}$ and $G_1^*$ is nearly $(k-1)$-regular by Claim \ref{C13}, and by Lemma \ref{jn} and (\ref{E2}), we have $\rho(G(r))\le\rho(A_r)<\rho(A_{\frac{1}{2}})$, a contradiction.
So $r=\frac{1}{2}$.
\end{proof}

By Claim \ref{Fca}, we complete the proof.
\end{proof}


\begin{thebibliography}{99}


\bibitem{BH} A.E. Brouwer, W.H. Haemers,
Spectra of Graphs, Springer, New York, 2012.

\bibitem{CLZ}
M. Chen, A.  Liu, X.  Zhang,
Spectral extremal results with forbidding linear forests,
Graphs Combin. 35 (2019)  335--351.

\bibitem{CLZ2}
M. Chen, A.  Liu, X.  Zhang,
On the spectral radius of graphs without a star forest,
Discrete Math. 344 (2021) 112269.



\bibitem{CDT}
S. Cioab\v{a}, D.N. Desai, M. Tait,
The spectral radius of graphs with no odd wheels,
European J. Combin. 99 (2022) 103420










\bibitem{CF}
S. Cioab\v{a}, L. Feng, M. Tait, X. Zhang,
The maximum spectral radius of graphs without friendship subgraphs,
Electron. J. Combin. 27 (2020) P4.22.


\bibitem{CH}
V. Chv\'{a}tal, D. Hanson,
Degrees and matching,
J. Combin. Theory Ser. B 20 (1976) 128--138.



\bibitem{DK}
D.N. Desai, L. Kang, Y. Li, Z. Ni, M. Tait, J. Wang,
Spectral extremal graphs for intersecting cliques,
Linear Algebra Appl. 644 (2022) 234--258.



\bibitem{FZL}
L. Fang, M. Zhai, H. Lin,
Spectral extremal problem on $t$ copies of $\ell$-cycle,
Electron. J. Combin. 31 (2024) 4.17.

\bibitem{FYZ} L. Feng, G. Yu, X. Zhang, Spectral radius of graphs with given matching number,
Linear Algebra Appl. 422 (2007) 133--138.

\bibitem{Fur}
Z. F\"uredi,
Tur\'an type problems, in: Surveys in Combinatorics, 1991 (Guildford, 1991),
London Math. Soc. Lecture Note Ser., 166, Cambridge Univ. Press, Cambridge, 1991,
pp. 253--300.

\bibitem{KST}
T. K\H{o}vari, V.T. S\'{o}s, P. Tur\'{a}n,
On a problem of K. Zarankiewicz,
Colloq. Math. 3 (1954) 50--57.


\bibitem{LND}
Y. Luo, Z. Ni, Y. Dong,
Spectral extremal graphs for disjoint odd wheels,
Linear Algebra Appl. 710 (2025) 243--266.


\bibitem{Mi}
H. Minc,
Nonnegative Matrices, John Wiley \& Sons, New York, 1988.

\bibitem{Mo} J. Moon, On independent complete subgraphs in a graph, Canad. J. Math. 20 (1968)
95--102.

\bibitem{NWK} Z. Ni, J. Wang, L. Kang, Spectral extremal graphs for disjoint cliques, Electron. J. Combin. 30 (2023) 1.20.







\bibitem{VN}
V. Nikiforov,
Stability for large forbidden subgraphs,
J. Graph Theory 62  (2009) 362--368.


\bibitem{NV2}
V. Nikiforov,
The spectral radius of graphs without paths and cycles of specified length,
Linear Algebra Appl. 432 (2010)  2243--2256.


\bibitem{Pik} O. Pikhurko, A note on the Tur\'an function of even cycles, Proc. Amer. Math. Soc. 140 (2012) 3687--3692.

\bibitem{Sim}
M. Simonovits, Extremal graph problems with symmetrical extremal graphs. Additional
chromatic conditions, Discrete Math. 7 (1974) 349--376.


\bibitem{MS}
M. Simonovits, A method for solving extremal problems in graph theory, stability problems, Theory of Graphs (Proc. Colloq., Tihany, 1966),  Academic Press, New York-London, 1968,  pp.~279--319.


\bibitem{T}
M. Tait,
The Colin de Verdi\`{e}re parameter, excluded minors, and the spectral radius,
J. Combin. Theory Ser. A 166 (2019) 42--58.

\bibitem{XZ} C. Xiao, O. Zamora,
A note on the Tur\'an number of disjoint union of wheels, Discrete Math. 344 (2021) 112570.

\bibitem{YLT}
L. Yuan, Extremal graphs for odd wheels, J. Graph Theory 98 (2021) 691--707.


\bibitem{YLY}
Q. Yuan, R. Liu, J. Yuan,
The spectral radius of $H_{2k}$-free graphs,
Appl. Math. Comput. 469 (2024) 128560.

\end{thebibliography}
\end{document}